\documentclass[11pt,reqno]{article}
\usepackage{amssymb,amsmath,amsthm}
\usepackage[utf8]{inputenc}
\usepackage{tikz}
\usepackage{hyperref}
\usepackage{xcolor}
\usepackage[backend=biber,maxbibnames=99]{biblatex}

\usepackage{ulem}
\usepackage[shortlabels]{enumitem}
\setlist[enumerate]{topsep=0pt,itemsep=-1ex,partopsep=1ex,parsep=1ex}
\usepackage{graphicx}

\newtheorem{statement}{}[section]
\newtheorem{theorem}[statement]{Theorem}
\newtheorem{lemma}[statement]{Lemma}
\newtheorem{proposition}[statement]{Proposition}
\newtheorem{definition}[statement]{Definition}
\newtheorem{corollary}[statement]{Corollary}
\newtheorem{remark}[statement]{Remark}
\newtheorem{example}[statement]{Example}

\makeatletter
\newcommand{\subjclass}[2][1991]{%
  \let\@oldtitle\@title%
  \gdef\@title{\@oldtitle\footnotetext{#1 \textbf{Mathematics subject classification:} #2}}%
}
\newcommand{\keywords}[1]{%
  \let\@@oldtitle\@title%
  \gdef\@title{\@@oldtitle\footnotetext{\textbf{Key words and phrases:} #1}}%
}
\makeatother

\def\face{\ensuremath\mathfrak{F}}
\def\HB{\ensuremath\mathcal{HB}}

 \def\Az{\ensuremath \mathfrak{C}_T} 
 
  \def\ex{\ensuremath \mathfrak{D}} 

\def\ucore{U-core} 

\def\dcup{\dot{\cup}}

\def\nn{{\|\cdot\|}}

\def\N{\Bbb N}
\def\IN{\hbox{{\rm I}\kern-.13em{\rm N}}}

\def\R{\mathbb{R}}
\def\IR{\hbox{{\rm I}\kern-.13em{\rm R}}}
\def\nin{n\in\N}

\def\co{{\rm conv}\,}
\def\cco{\overline{\rm conv}\,}
\def\sp{\hbox{{\rm span }}}

\def\ee{\varepsilon}

\def\aa{\alpha}
\def\dd{\delta}

\def\eor{\hfill{\circledR}}

\def\Ext{\hbox{\rm{Ext} }}
\def\NA{\hbox{\rm{NA}}}
\def\JB{\hbox{\rm{JB}}}

\def\la{\langle}
\def\ra{\rangle}

\def\wh{\widehat}
\def\ol{\overline}

\title{Phelps Property U and $C(K)$ spaces}

\subjclass[\textbf{2020}]{46B20, 46B04, (46A22, 46E15)}

\author{Ch. Cobollo,
 A.J. Guirao,
 and V. Montesinos
}

\date{}

\begin{document}
\normalem
\maketitle
\begin{abstract}
A subspace $X$ of a Banach space $Y$ has {\em Property U}  whenever every continuous linear functional on $X$ has a unique norm-preserving (i.e., Hahn--Banach) extension to $Y$ (Phelps, 1960). Throughout this document we introduce and develop a systematic study of the existence of {\em U-embeddings} between Banach spaces $X$ and $Y$, that is, isometric embeddings of $X$ into $Y$ whose ranges have property U. In particular, we are interested in the case that $Y=C(K)$, where $K$ is a compact Hausdorff topological space. We provide results for general Banach spaces and for some specific set-ups, such as $X$ being a  finite-dimensional space or a $C(K)$-space.

\end{abstract}

\section{Preliminaries}

Given a Banach space $Y$ and a closed subspace $X$, let us denote by ${i\colon X \to Y}$ the inclusion (linear isometric embedding) from $X$ into $Y$. Its adjoint ${i^*\colon Y^*\rightarrow X^*}$ can be identified to the canonical quotient mapping $q\colon Y^*\rightarrow Y^*/X^{\perp}$, as $X^*$ is isometrically isomorphic to $Y^*/X^{\perp}$. The mapping $q$ acts as a ``restriction operator'', sending linear continuous functionals $y^*\in Y^*$ to its restriction to $X$, i.e., $q(y^*)={y^*}|_{X}$ for $y^*\in Y^*$. The mapping $q$ is $\nn$-$\nn$-continuous, $w$-$w$-continuous, and $w^*$-$w^*$-continuous.

\medskip

Along the paper, and if nothing is said to the contrary, {\em if $X$ is a Banach space, then its topological dual will always be considered endowed with its $w^*$-topology}.

\medskip

Given $x^*\in X^*$, the set of all continuous linear extensions of $x^*$ to $Y$ is $q^{-1}(x^*)$. If $X\not=Y$, this set contains functionals with greater norm than $\|x^*\|$). Observe that $q^{-1}(x^*)=y^*+X^{\perp}$, where $y^*$ is a particular continuous linear extension of $x^*$ to $Y$. Thus, the (nonempty) set of all Hahn--Banach (i.e., norm-preserving) extensions  of $x^*$ to $Y$ is
\begin{equation}\label{eq-HB(x^*)}
\HB(x^*):=q^{-1}(x^*)\cap \|x^*\| S_{Y^*} = (y^* + X^\perp)\cap \|x^*\| S_{Y^*}.
\end{equation}

More generally, the set of all Hahn--Banach extensions of the elements in a subset $A\subset X^*$ will be denoted by  $\HB(A)$, i.e., 
\begin{equation*}
    \HB(A):= \{y^*\in Y^* : y^*\in \HB(x^*) \text{ for some }x^*\in A \}
\end{equation*}
 By the word ``subspace'' we shall mean, if nothing is said to the contrary, a {\em closed} linear subspace.
 
 \medskip
 
 The following is the central concept treated here:

\begin{definition}
\label{def-property-u}
Let $X$ be a subspace of a Banach space $Y$. It is said that $X$ has {\bf property U in $Y$} whenever every $x^*\in X^*$ has a unique Hahn--Banach  extension to $Y$.
\end{definition}
 In other words, for every $x^*\in S_{X^*}$, the set $\HB(x^*)$ is a singleton. This property was introduced by Phelps in \cite{Phelps1960a}. He started a systematic study of property U within certain classical Banach spaces. Later, Smith and Sullivan considered the particular case $Y:=X^{**}$.

It was in 1939 when Taylor ---see~\cite{Taylor1939}--- studied the question of identifying which Banach spaces $Y$ satisfy that any of its subspaces has property U in $Y$. The complete characterization was provided in 1958 by Foguel. We include it here for later reference: 

\begin{proposition}[Taylor--Foguel\rm{, \cite{Foguel1958}}]\label{prop-rotund-and-U}
Let $(Y,\nn)$ be a Banach space. The dual norm $\nn^*$ on $Y^*$ is strictly convex if, and only if, every subspace of $Y$ has property U in $Y$.
\end{proposition}
\begin{proof} One direction is given by the following observation: Given $x^*\in S_{X^*}$, the set $\HB(x^*)=q^{-1}(x^*)\cap S_{Y^*}$ is a nonempty ``flat'' subset of $S_{Y^*}$, so it reduces to a point. For the reverse implication, assume that the sphere $S_{Y^*}$ contains a non-trivial segment $[y_1^*,y_2^*]$. Take $X:=\{y\in Y : \la y, y_1^*-y_2^*\ra=0\}$. Thus, $X^\perp = \sp(y_1^*-y_2^*)$. Take $x^*_0:= q(y_2^*)$. Therefore, $\HB(x^*_0)= (y_2^*+X^\perp)\cap S_{Y^*}$, a set that contains the segment $[y_1^*,y_2^*]$. Thus, $S_{Y^*}$ is not strictly convex.
\end{proof}

\begin{remark}\label{rem-q1}
\rm Assume that $X$ is a subspace of a Banach space $Y$ and that $X$ has property U in $Y$. As it is well known and easy to see, every element $e^*_1\in \Ext B_{X^*}$ is the image of an element $e^*_0\in \Ext  B_{Y^*}$.  This element $e^*_0$ is unique. This shows that $\HB(\Ext B_{X^*}) \subset \Ext B_{Y^*}$.\eor
\end{remark}
We can complete the information given in Remark \ref{rem-q1}. First we need an almost trivial observation about mappings. Let us record it as a lemma for later reference. 

\begin{lemma}
\label{lemma-intersection}
Let $A$ and $B$ be two nonempty sets, and let $f:A\rightarrow B$ be a mapping. Let $A_0\subset A$ and $B_0\subset B$. Then $f(A_0\cap f^{-1}(B_0))=f(A_0)\cap B_0.$
\end{lemma}

\begin{proposition}\label{prop-U-homeo}
Let $Y$ be a Banach space and let $X$ be a linear subspace of $Y$. Then, $X$ has property U in $Y$ if, and only if, $q|_{\HB(S_{X^*})}:\HB(S_{X^*})\rightarrow S_{X^*}$ is (onto and) one-to-one. If this is the case, it is a $w^*$-$w^*$-homeomorphism. 
\end{proposition}
\begin{proof} The necessary and sufficient condition is just the definition of property U. That the mapping is $w^*$-$w^*$-continuous was mentioned above. What is somehow interesting is the fact that the inverse mapping ---that exists under property U--- is in fact $w^*$-$w^*$-continuous, too. 

For this, it is enough to show that $q$ sends $w^*$-relatively closed sets in $\HB(S_{X^*})$ to $w^*$-relatively closed sets in $S_{X^*}$. Indeed, let $B$ be an arbitrary $w^*$-relatively closed set of $\HB(S_{X^*})$. Thus, $B=C\cap \HB(S_{X^*})$, where $C$ is a $w^*$-closed (hence $w^*$-compact) subset of $B_{Y^*}$. Observe that if $C\subset B_{Y^*}$, then $C\cap q^{-1}(S_{X^*})=C\cap \HB(S_{X^*})$.  Then, by Lemma \ref{lemma-intersection}, $q(B)=q(C)\cap S_{X^*}$. Since $q(C)$ is also $w^*$-compact (hence $w^*$-closed), this concludes the proof. 
\end{proof}

\subsection{Norm-attaining functionals: wU property}

Assume again that $Y$ is a Banach space and that $X$ is a subspace of $Y$. Denote by 
$$
\NA(X) \hbox{ the subset of $X^*$ of all norm-attaining functionals.}
$$
If $x^*_0\in \NA(X)\cap S_{X^*}$ and $x_0\in S_X$ is such that $\langle x_0,x^*_0\rangle=1$, then
\begin{equation}\label{eq-qinv}
q^{-1}(x^*_0)\subset \{y^*\in Y^*:\langle x_0,y^*\rangle=1\}.
\end{equation}
Notice that if $y^*$ is a Hahn--Banach extension of $x_0^*$, then $y^*$ also attains its norm (at $x_0$). 

In general, the two sets $q^{-1}(x^*_0)$ and $\{y^*\in Y^*:\ \langle x_0,y^*\rangle=1\}$ are different: the first one consists of all the extensions of $x_0^*$ to $Y$, i.e., $y^*+X^{\perp}$, where $y^*$ is a given extension of $x^*_0$ to $Y$, while the second one is the affine hyperplane in $Y^*$ defined by $x_0$ ---to be precise, the affine hyperplane defined by $i(x_0)\in Y$.

The set 
$$
\face(x_0):=\{x^*\in S_{X^*}:\ \langle x_0,x^*\rangle=1\}
$$ 
is the {\bf face} of $B_{X^*}$ defined by $x_0$. Accordingly, $\face(i(x_0))$ is
the face of $B_{Y^*}$ defined by $i(x_0)\ (\in Y)$. By equation (\ref{eq-qinv}), 
$\HB(x_0^*)=q^{-1}(x^*_0)\cap S_{Y^*}\subset \face(i(x_0))$. It is easy to see 
that $\face(i(x_0))=q^{-1}(\face(x_0))\cap 
S_{Y^*}$ ---equivalently, $\face(i(x_0))= \HB(\face(x_0))$---, 
and that $q(\face(i(x_0))=\face(x_0)$. Summing up, the faces of $B_{Y^*}$ consist of the Hahn--Banach extensions of the elements in the faces of $B_{X^*}$, and conversely, the elements in the faces of $B_{X^*}$ are just the restrictions of the elements in the faces of  $B_{Y^*}$.

\medskip

The following is a weakening of property U:  

\begin{definition}\label{def-property-wu}
Let $X$ be a subspace of a Banach space $Y$. The space $X$ has {\bf property wU} {\em  ({\bf weak U}) } in $Y$ (or {\bf $X$ is wU-embedded in $Y$})  whenever every functional $x^*\in \NA(X)$ has a unique Hahn--Banach extension to $Y$.
\end{definition} In other words, for every $x_0\in S_{X}$  and $x^*_0\in \face(x_0)$ there exists a unique $y^*_0\in \face(i(x_0))$ such that $q(y^*_0)=x^*_0$. Thus, in this case $q|_{\face(i(x_0))}$ is a $w^*$-$w^*$-homeomorphism from $\face(i(x_0))$ onto $\face(x_0)$ (see the argument in Proposition \ref{prop-U-homeo} above) ---and this is equivalent to the wU property of $X$.

Let us record the previous simple observation as a proposition for later reference. Compare with Proposition \ref{prop-U-homeo} above.

\begin{proposition}\label{prop-wu-and-faces}
Let $X$ be a subspace of a Banach space $Y$. The space $X$ has property wU in $Y$ if, and only if, $q|_{\HB(NA(X)\cap S_{X^*})}:\HB(NA(X)\cap S_{X^*})\rightarrow NA(X)\cap S_{X^*}$ is one-to-one. If this is the case,  $q|_{\HB(NA(X)\cap S_{X^*})}$ is a $w^*$-$w^*$-homeomorphism from  $\HB(NA(X)\cap S_{X^*})$ onto $NA(X)\cap S_{X^*}$. In particular, this happens if, and only if, 
$q|_{\face(i(x))}$ is a $w^*$-$w^*$-homeomorphism from $\face(i(x))$ onto $\face(x)$, for every $x\in S_X$.
\end{proposition}

It is a crucial fact that, in presence of property wU, the $w^*$-$w^*$-homeo\-morphisms between faces are given by the restrictions of the linear operator $q$. This means that somehow the affine structure of the faces $\face(i(x_0))$ and $\face(x_0)$ is preserved (see Proposition \ref{prop-wu-and-facesT^*} below). It motivates the following definition and corollaries.  

\begin{definition}
Let $X$ be a Banach space and let $x\in S_X$. We say that $x$ is a {\bf  point of $n$-almost G\^ateaux smoothness} whenever the subdifferential of the norm at $x$ ---in our notation, $\face(x)$--- has affine dimension $n$. In particular, $0$-almost G\^ateaux smoothness is just G\^ateaux smoothness.
\end{definition}

\begin{corollary}
A wU-embedding sends points of $n$-almost G\^ateaux smoothness to points of $n$-almost G\^ateaux smoothness.
\end{corollary}


Indeed, we can be a bit more precise. As we said above, it is not just the affine dimension of the faces, but the whole affine structure which is preserved. This is remarkable, since even when $n=2$ (2-dimensional faces), we can find faces that are circles, squares or triangles. Those shapes have the same affine dimension but a very different extremal structure. Also, a special characteristic of the triangular shape is that every point in the triangle can be written as a unique convex combination of its extreme points. We will see that the triangular ---indeed, the simplicial--- case will play a relevant r\^ole throughout the document.

\begin{remark}
\label{rem:C(K)-simplex}
\rm
 For instance if $Y=C(K)$ for some compact Hausdorff space $K$, the face determined by an $n$-almost G\^ateaux smooth point would indeed be an $n$-simplex ---since the set of extreme points of $B_{C(K)^*}$ is affine independent. Therefore, we can only hope to U-embbed Banach spaces $X$ whose $n$-almost G\^ateaux smooth points  define faces linearly-convex isomorphic to an $n$-simplex ---In general, the face determined by an $n$-almost G\^ateaux smooth point could have affine dimension $n$ but don't even be a polytope. For $n=1$  there is no such ambiguity ---indeed, this leads to a necessary condition for U-embeddings in $C(K)$-spaces, as we will see later in Proposition \ref{prop-wU-simplexoid}. \eor
\end{remark}

Needless to say, if $X$ is a finite-dimensional subspace of a Banach space $Y$, wU-embedding and U-embedding of $X$ in $Y$ are identical notions.  Notice that when $X^*$ has dimension $1$, $B_{X^*}$ is a segment with two end points $x^*_0$ and $-x^*_0$ (those points are its faces), and the point $x_0$ that norms $x^*_0$ has a face in $Y^*$ that reduces to a point. Thus, we have the following statement.

\begin{corollary} \label{coro-one-dimension}
A Banach space $Y$ contains a U-embedded one-dimensional subspace if, and only if, $Y$ has at least a G\^ateaux-smooth point.
\end{corollary}

Indeed, the relation between wU property and G\^ateaux differentiability is even closer. Recall that Proposition \ref{prop-rotund-and-U} above reveals that having property U in every subspace of $Y$ is equivalent to strict convexity of $B_{Y^*}$. So, when we restrict the attention to the norm-attaining functionals, G\^ateaux smoothness appears as the wU counterpart of strict convexity on the dual (compare with Proposition \ref{prop-rotund-and-U}):

\begin{proposition}\label{prop-gateaux} Let $Y$ be a Banach space. Its norm is G\^ateaux differentiable if, and only if, every  subspace $X$ of $Y$ has property wU in $Y$.
\end{proposition}
\begin{proof} Assume first that $\nn$ on $Y$ is G\^ateaux differentiable. Let $X$ be a subspace of $Y$. Given $x^*_0\in NA(X)\cap S_{X^*}$ that attains its norm at $x_0\in S_{X}$, we have ---see (\ref{eq-HB(x^*)}) and (\ref{eq-qinv}) above--- that
\begin{equation}\label{eq-inclusion}
\HB(x^*_0)=q^{-1}(x^*_0)\cap S_{Y^*}\subset \face(i(x_0)),
\end{equation}
The last set in (\ref{eq-inclusion}) is a singleton due to the G\^ateaux differentiability of the norm at $i(x_0)$. This shows the necessity.  For the sufficiency, it is enough to repeat the argument in the proof of the Taylor--Foguel Proposition \ref{prop-rotund-and-U}. 
\end{proof}

In Remark \ref{rem:C(K)-simplex} we could already appreciate one of the particularities when embedding in $C(K)$ spaces, which is the main topic of the paper. The structural behavior of this kind of spaces will lead us to develop some specific tools, but it is worth to note that without further development, we can recover one of Phelps' results for the $C(K)$ case.

\begin{corollary}[Phelps, {\cite[Theorem 3.2]{Phelps1960a}}] Let $K$ be a compact Hausdorff space, and let $X$ be a finite-dimensional subspace of $(C(K),\nn_{\infty})$. Let $\dim(X)=n$. Assume that $X$ has property U in $C(K)$. Then every $x^*$ in $S_{X^*}$ is the convex combination of at most $n$ elements $\pm q(\dd_k)$, $k\in K$. In particular, for every $x\in X$ it holds that $\|x\|=|x(k)|$ for at most $n$ points $k\in K$.
\end{corollary}
\begin{proof} It is enough to apply Proposition \ref{prop-wu-and-faces} and the classical  Minkowski--Carath\'eodory result on the  convex hull of a finite-dimensional set, together with the fact that the set of  extreme points of $S_{C(K)^*}$ is $\{\pm\dd_{k}:\ k\in K\}$.
\end{proof}

The other main result for  $C(K)$ spaces provided by Phelps in \cite{Phelps1960a} is the following. Let $K$ be a compact space and $Z\subset K$ a closed subset. The {\bf closed ideal of $C(K)$ defined by $Z$} is the linear subspace $I_Z(K):= \{f\in C(K) : f|_{Z}=0\}\ (\subset C(K))$.

\begin{proposition}[Phelps, {\cite[Lemma 3.1]{Phelps1960a}}]
\label{prop:phelps-ideals}
Let $K$ be a compact space and $Z\subset K$ a closed subset. Then, the closed ideal  $I_Z(K)$ defined by $Z$ has property U in $C(K)$.
\end{proposition}

Phelps provided a straightforward proof of Proposition \ref{prop:phelps-ideals}. Nowadays, it can be obtained as a consequence of the study of M-ideals, since those have property U, and the M-ideals in $C(K)$ spaces are precisely the closed ideals of $C(K)$ ---see \cite{Harmand1993}.

\section{U- and wU-embeddings}\label{sec:U-embeddings}
Provided a Banach space $X$, we wonder about the existence of an embedding from $X$ to a Banach space $Y$ whose range has property U in $Y$. Motivated by that, we introduce the following definition:

\begin{definition}
Let $X$ and $Y$ be Banach spaces. Let $T\colon  X\to Y$ be a linear isometry into. Given $y^*\in Y^*$ and $x^*\in X^*$, we say that $y^*$ is an \textbf{extension of 
$\boldsymbol{x^*}$ through $\boldsymbol{T}$} whenever $T^*(y^*)=x^*$. 
If additionally $\|y^*\|=\|x^*\|$ we will say that {\bf $y^*$ is a Hahn-Banach 
extension of $x^*$ (through $T$)}. We say that $T$ is a \textbf{U-embedding} 
whenever every $x^*\in S_{X^*}$ has a unique Hahn--Banach extension through 
$T$. 
\end{definition}

The definition above is motivated by the fact that whenever $T^*(y^*)=x^*$, then  $y^*$ is an actual extension to $Y$ of the functional $f$ defined on $T(X)$ by $f(T(x))=\langle x,x^*\rangle$ for every $x\in X$. Indeed, 
\begin{equation}
\la y^*,T(x)\ra= (y^* \circ T)(x)= T^*(y^*)(x)=\la x^*,x\ra = f(T\big(x)\big).
\end{equation}
In this case, we say that $f$ {\bf represents $x^*$ throughout} $T$.

An equivalent definition of U-embedding would be a linear isometry into ${T\colon X\to Y}$ whose range has property U in $Y$. Notice also that, provided a linear isometry $T\colon X\to Y$, it is natural to identify $X$ with the range of the map $T(X)\subset Y$, and think of $X$ as a linear subspace of $Y$. Formally, we have $X \xrightarrow{T} T(X)\xrightarrow{i} Y $. Thus, defining $\phi\colon X \to Y$ as $ i\circ T$, we have $\phi^*=T^* \circ i^*=T^* \circ q$. As $T$ is an isometric isomorphism,
all the results provided in the previous section for the inclusion and its adjoint are still valid to any linear isometry and its adjoint, and viceversa.

\medskip

Following with the notation, provided $T\colon X\to Y$ a linear isometry and given $x^*\in X^*$, we denote the set of its Hahn--Banach extensions through $T$ by 
\begin{equation*}
    \HB_T (x^*):= (T^*)^{-1}(x^*) \cap \|x^*\| S_{Y^*} = ((T^*)^{-1}(x^*) + T(X)^\perp)\cap \|x^*\| S_{Y^*}.
\end{equation*}

Obviously, $T\colon X\to Y$ is a U-embedding if and only if $\HB_T(x^*)$ is a singleton for every $x^*\in X^*$. All the other concepts and those related to the wU-embeddings are defined analogously as in the previous section.

\begin{remark}\label{rem:HB-symm}
\rm
Analogously to the previous section, given a linear isometry $T\colon X\to Y$ and a subset $A\subset X^*$, we define its set of Hahn--Banach extensions as:
\begin{equation*}
    \HB_T(A):= \{y^*\in Y^* : y^*\in \HB_T(x^*) \text{ for some }x^*\in A \}.
\end{equation*}
Despite being almost trivial, it is worth to state that if the set $A\subset X^*$ is symmetric, so it is $\HB_T(A)$.
The symmetry arguments will play a key r\^ole  when embedding into $C(K)$ spaces.\eor 
\end{remark}

 For later reference, we state the isometric version of Propositions \ref{prop-U-homeo} and \ref{prop-wu-and-faces}, both having analogous proofs to the original ones.

\begin{proposition}\label{prop:U_homeoT^*}
 Let $T\colon X\to Y$ be a continuous linear mapping. Then, $T$ is a U-embedding if, and only if,  ${T^*}|_{\HB_T(S_{X^*})}\colon  \HB_T(S_{X^*}) \to S_{X^*}$ is a $w^*$-$w^*$-homeomorphism.
\end{proposition}

\begin{proposition}\label{prop-wu-and-facesT^*}
Let $T\colon X\to Y$ be a linear isometry. Then, $T$ is a wU-embedding if, and only if, ${T^*}_{|\face(T(x))}$ is a $w^*$-$w^*$-homeomorphism from $\face(T(x))$ onto $\face(x)$, for every $x\in S_X$.
\end{proposition}

We will end this section by making some comments about composition of U-embeddings. The following result has an almost trivial proof, that we shall omit.

\begin{lemma}\label{lemma-chain} Let $i\colon X\to Y$ and $j\colon Y\to Z$ be isometries into. Assume that $j$ is a U-embedding. Then, the composition $j\circ i$ is a U-embedding if, and only if,  $i$ is a U-embedding.
\end{lemma}

Notice that if $j\circ i$ is a U-embedding, so it is $i$; however, it is not possible to conclude that $j$ should be a U-embedding, too (arguing in terms of subspaces, an element $y^*$ can have two distinct Hahn--Banach extensions to $Z$ that coincide on $X$). In the other direction, the previous lemma concludes that if $i$ and $j$ are both U-embeddings, so it is $j\circ i$.

\section{Isometries into $C(K)$ and first consequences}
\label{subsect-intro-c(K)}

\subsection{A specific set-up for the $C(K)$ case}

Let $K$ be a compact Hausdorff topological space. We shall consider $C(K)$, the space of all continuous real-valued functions defined on $K$, endowed with its maximum norm $\nn_{\infty}$. Put $M(K):=C(K)^*$. The space $M(K)$ consists of all regular Borel measures on $K$. The norm $\nn_{\infty}^*$ on $M(K)$ is the total variation of a measure. For $k\in K$, let $\dd_k^K\in M(K)$ be the corresponding Dirac delta. Define a mapping $\delta^K\colon K\to M(K)$ as $\delta^K(k)=\delta_k^K$ for $k\in K$. Then $\delta K:=\dd^K(K)=\{\dd_k^K: k\in K\}$ in its $w^*$-topology is homeomorphic to $K$. We shall often identify the two topological spaces $K$ and $(\dd K,w^*)$. We shall write $\dd_k$  instead of $\dd^{K}_{k}$ (and $\dd$ instead of $\dd^K$) if there is no risk of misunderstanding.

We are mainly concerned with those closed subspaces of $C(K)$ having property U (or wU)  in $C(K)$. Our approach here focuses however in those linear isometric embeddings $T\colon X\to C(K)$ for which $T(X)$ has property U (or wU) in $C(K)$. The following observation inspired the majority of this section.

\begin{remark}\label{rem-C-of-the-dual-ball}
\rm A Banach space can be naturally embedded in a space of the form $C(K)$ (namely, $X$ is linearly isometric to a subspace of $C((B_{X^*},w^*))$ by means of the mapping $j\colon X\rightarrow C((B_{X^*},w^*))$ given by $j(x):=x|_{B_{X^*}}$ for all $x\in X$). The question on U-embeddability  translates, in particular, to this case. However, this is just the uninteresting one, since $X$ has not property U (neither property wU) in $C((B_{X^*},w^*))$, even in a strong way: no $x^*\in S_{X^*}$ has a unique Hahn--Banach extension to $C((B_{X^*},w^*))$. The reason is that, in case $x^*\in S_{X^*}$, both $\dd_{x^*}$ and $-\dd_{-x^*}$ are clearly two distinct Hahn--Banach extensions of $x^*$ to $C((B_{X^*},w^*))$.\eor
\end{remark}

\medskip

 As Remark \ref{rem-C-of-the-dual-ball} above shows, the ``canonical'' way of embedding a Banach space into a $C(K)$ never allows $X$ to have unique extensions to it. However, there are many other ways to find an embedding from $X$ into a $C(K)$ space. We will appeal ---see the next paragraph--- to a classical result that correlates bounded linear operators from $X$ to $C(K)$ and continuous maps from $K$ to $X^*$. 

\medskip

It is classical ---see \cite[Theorem VI.7.1]{ds}--- that given a Banach space $X$, there exists a linear  isometry 
from $B(X,C(K))$ (i.e., the Banach space of bounded linear operators from $X$ to $C(K)$), onto $C(K,(X^*,w^*))$. Precisely, if ${T\colon X\rightarrow C(K)}$ is a continuous linear mapping, let 
\begin{equation}
\label{eq-varphi}
\boxed{F_T:=T^*\circ \delta^K:K\rightarrow X^*.} 
\end{equation}
On the other hand, if $F\colon K \rightarrow (X^*,w^*)$ is a continuous map, put

\begin{equation}\label{eq-phi}
\big(T_F(x)\big)(k):= \langle x,F(k)\rangle, \text{ for $x \in X$ and $k\in K$}.
\end{equation}
Thus, 
$$
\boxed{
T_F:X\rightarrow C(K).}
$$
Notice that $F_{T_F}=F$.

Observe that $\|F_T\|_{\infty}=\sup_{\dd  K}T^*=\sup_{B(M(K))}T^*=\|T^*\|=\|T\|$, due to the fact that $T^*$ is $w^*$-$w^*$-continuous and that $B_{M(K)}=\cco^{w^*}(\pm\dd K$) (where $\pm\dd K:=\dd K\cup(-\dd K)$). Thus, the map 
$T \mapsto F_T$ is indeed an isometric isomorphism from $B(X,C(K))$ onto $C(K,(X^*,w^*))$.

\medskip

Let $X$ be a Banach space. A subset $\JB\subset S_{X^*}$ is said to be a {\bf James boundary for $X$} if every $x\in S_X$ attains its norm at some point of $\JB$. For example, the set $\Ext B_{X^*}$ of all extreme points of $B_{X^*}$, is a James boundary for $X$. The Separation Theorem shows that $\cco^{w^*}(\JB)=B_{X^*}$ for every James boundary $\JB$. A subset $N$ of $B_{X^*}$ is said to be {\bf $1$-norming} if $\sup_{x^*\in N}|\langle x,x^*\rangle|=1$ for every $x\in S_X$. Clearly, every James boundary is $1$-norming, and every $1$-norming set is $w^*$-linearly dense in $X^*$.

A simple observation is that, 
in case that $T\colon X \to Y$ is an isometry into and $\JB\subset S_{Y^*}$ is a James boundary for $Y$, then $T^*(\JB)\cap S_{X^*}$ is a James boundary for $X$. If $N$ is a $1$-norming subset of $B_{Y^*}$, then $T^*(N)$ is a $1$-norming subset of $B_{X^*}$.

\begin{lemma}\label{lemma-isometry} Let $K$ be a compact Hausdorff space, and let $X$ be a Banach space.
Let $F\in C(K,(X^*,w^*))$, and let $T_F:X\rightarrow C(K)$ be its associated operator defined in {\em (\ref{eq-phi})}. Then, the following are equivalent:
\begin{enumerate}[\rm (i)]
    \item $T_F$ is an isometry into.
    \item $ F(K) \cup -F(K) \ (\subset B_{X^*})$ is a James boundary for $X$.
    \item $F(K) \ (\subset B_{X^*})$ is $1$-norming.
    \item 
   $\Ext B_{X^*} \subset  F(K)\cup -F(K) \subset B_{X^*}$.
\end{enumerate}

\end{lemma}
\begin{proof}
(i)$\Rightarrow$(ii) If (i) holds, then $(T_{F})^*$ is a quotient mapping. Its restriction to $\dd K$ is just $F$ (see equation (\ref{eq-varphi}) above). Since $\dd K\cup (-\dd K)$ is a James boundary for $C(K)$, previous observations show that $ F(K)\cup -F(K)$ is a James boundary for $X$.

(ii)$\Rightarrow$(iii) is always true.

(iii)$\Rightarrow$(iv) is a consequence of the fact that $\cco^{w^*}( F(K)\cup -F(K))=B_{X^*}$ and Milman's converse to the classical Krein--Milman Theorem (see, e.g., \cite[Theorem 3.66]{y2}).

(iii)$\Rightarrow$(i) is clear, and (iv)$\Rightarrow$(iii) follows from the fact that $\Ext B_{X^*}$ is $1$-norming.
\end{proof}

\begin{remark}\label{rem-core}\rm 
Let $T:X\rightarrow C(K)$ be a linear isometric embedding. Put $K_0:=F_T(K)\ (\subset B_{X^*})$, and endow it with the restriction of the $w^*$-topology. According to Lemma \ref{lemma-isometry}, $\pm K_0$ is a James boundary for $X$. In particular  $\cco^{w^*}(\pm K_0)=B_{X^*}$ (after all, $\Ext(B_{X^*})\subset \pm K_0$, and $K_0$ is compact). Put $C(f):=f\circ F_T$ for $f\in C(K_0)$. This defines a linear mapping $C:C(K_0)\rightarrow C(K)$. Observe that 
\begin{eqnarray*}
\lefteqn{\|C(f)\|_{\infty}=\sup_{k\in K}|C(f)(k)|}\\
&&=\sup_{k\in K}|f\circ T^*\circ \dd^K(k)|=\sup_{k_0\in K_0}|f(k_0)|=\|f\|_{\infty},
\end{eqnarray*}
This shows that $C:C(K_0)\rightarrow C(K)$ is a linear isometry. Let $R:X\rightarrow C(K_0)$ be defined as $R(x):=x|_{K_0}$. Thus, $R$ is a linear isometry from $X$ into $C(K_0)$. Obviously, $T=C\circ R$. If $T$ is a U-embedding, so it is $R$.  \eor
\end{remark}

\begin{remark}\label{rem:delta-extension}
\rm
Given $F\in C(K,(X^*,w^*))$ such that $T_F$ is an isometry (i.e., $F$ satisfies any of the equivalent conditions (ii), (iii), or (iv) in Lemma \ref{lemma-isometry}) and $k\in K$ is such that $\|F(k)\|=1$, then $\delta_k\in C(K)^*$ is a Hahn--Banach extension of $F(k)$ through $T_F$. \eor 
\end{remark}

\subsection{The \ucore}\label{subsec:U-core}
The present subsection will be devoted to justify the relevance of the following notion.

\begin{definition}
\label{def:U-core}
Let $T\colon X\to C(K)$ be a linear isometry. We define its \ucore\ as the set 
\begin{equation*}
   \boxed{\Az:= F_T^{-1}(\Ext B_{X^*}) \subset K.}
\end{equation*}

\end{definition}

\rm{}

 Along this note, we will use the symbol $A\dcup B $ to denote $A\cup B$ whenever we also know that $A\cap B= \emptyset$.
 
 \medskip
 
Let us list some easy properties of the \ucore\ for later reference. We omit the proofs.

\medskip

(i) ${\delta (\Az)= \HB_T(\Ext B_{X^*}) \cap \dd(K)}$.

(ii) As $T$ is an isometry, item (iv) of Lemma  \ref{lemma-isometry} implies that $\Az$, and so $\dd(\Az)$, are both non-empty. Indeed, every extreme point of $B_{X^*}$ is the image under $T^*$ of an extreme point of $M(K)$, i.e., a point of the form $\dd_{k}$ or $-\dd_{k}$, for some $k\in K$. As $B_{X^*}$ is symmetric, the result follows.

(iii) Since $\dd(\Az)\subset K$, it follows that $\ol{\dd (\Az)}^{w^*}\cap \ol{- \dd (\Az)}^{w^*}=\emptyset$.

(iv) Put
$$
\boxed{
\hbox{Ext}^+:=T^*(\dd(\Az))\ (=F_T(\Az)),}
$$
and  
$$
\boxed{
\hbox{Ext}^-:=T^*(\dd(-\Az))\ (=-F_T(\Az)).
}
$$
Notice that $\hbox{Ext}^-=-\hbox{Ext}^+$, and that $\Ext B_{X^*}= \hbox{Ext}^+\cup\hbox{Ext}^-$. 

(v) $F_{T}(K)\cap\Ext B_{X^*}=\hbox{Ext}^+$.
 
(vi) Since $T^*$ is $w^*$-$w^*$-continuous, $T^*\left(\ol{\HB_T(\Ext B_{X^*})}^{w^*}\right)= \ol{\Ext B_{X^*}}^{w^*}$, and $ F_T\left(\ol{\Az}\right)= \ol{\hbox{Ext}^+}^{w^*}$.

(vii) From (iv) and (vi) it follows that 
$$ \ol{\Ext B_{X^*}}^{w^*} =\ol{\hbox{Ext}^+}^{w^*}\cup -\ol{\hbox{Ext}^+}^{w^*}\ (=F_T(\ol{\Az})\cup -F_T(\ol{\Az})).
$$

(viii) If $T$ is a U-embedding, then $T^*|_{\HB_T(S_X^*)}:\HB_T(S_{X^*})\rightarrow S_{X^*}$ is a $w^*$-$w^*$-homeomorphism (see Proposition \ref{prop:U_homeoT^*}), hence $\hbox{Ext}^+\cap \hbox{Ext}^-=\emptyset$. Thus, 
$\Ext B_{X^*}=\hbox{Ext}^+ \ \dcup \ \hbox{Ext}^-$.

\subsection{Face structure of the dual of $C(K)$-spaces}
{\label{subsec:face-structure}

Fix a compact Hausdorff space $K$. The Riesz representation theorem gives a handful description of the face $\face(f)$ determined by any function $f\in S_{C(K)}$. This is the content of Lemma \ref{lemma_faces_CK} below. A measure $\mu$ on $K$ is said to be {\bf supported on a Borel set $B\subset K$} whenever $\mu(C)=0$ for every closed $C\subset K$ disjoint with $B$. Given $f\in C(K)$, put 
$$
\boxed{{\sigma^+(f):=\{t\in K: \ f(t)=1\}},}
$$
and 
$$
\boxed{
{\sigma^-(f):=\{t\in K: \ f(t)=-1\}}.}
$$  

\begin{lemma}[{\cite{Aviles2014}}]\label{lemma_faces_CK}
Let $K$ be a compact Hausdorff space. Fix $f\in S_{C(K)}$. A measure $\mu\in S_{M(K)}$ belongs to the face $\face(f)$ if, and only, if its positive part $\mu^+$ is supported on $\sigma^+(f)$ and its negative part $\mu^-$ is supported on $\sigma^-(f)$.
\end{lemma}

Given a space $X$ and a point $x\in S_X$, it is simple to show that the extreme points of the face $\face(x)$ are precisely the extreme points of $B_{X^*}$ sitting in $\face(x)$. In particular, if $X=C(K)$ for a compact topological space $K$, and $f\in S_{C(K)}$, then
\begin{equation}
\label{eq-extreme-face-M(K)}
\Ext\face(f)=\delta^K\big(\sigma^+(f)\big)\dcup\Big(-\delta^K\big(\sigma^-(f)\big)\Big).
\end{equation}
Notice that each of the sets in the disjoint union in (\ref{eq-extreme-face-M(K)}) is a closed set. If there exists an linear isometric embedding $T$ of a space $X$ into a space $C(K)$, we get in particular that, for $x\in S_X$, 
\begin{equation}\label{eq-extreme-faces}
\begin{split}
     \Ext \face (Tx) &= \{\dd_t : \la Tx,\dd_t \ra=1\} \dcup \{-\dd_t : \la Tx,-\dd_t \ra=1\} \\
      & = \{\dd_t : |\la Tx,\dd_t \ra|=1\} \\
       & = \dd^K (\sigma^+(Tx))\dcup -\dd^K (\sigma^-(Tx)).
\end{split}
\end{equation}
It is easy to see that
\[
\sigma^+(Tx)= F_T^{-1}(\face(x))\text{, and }\sigma^-(Tx)=F_T^{-1}(\face(-x)).
\]

\begin{remark}\label{rem-extreme-in-faces-closed}
\rm Assume that $X$ is a Banach space wU-embedded in a space $C(K)$, for a compact topological space $K$. If $x\in X$, it follows from (\ref{eq-extreme-faces}) above that $\Ext \face (Tx)$ is $w^*$-closed. In view of Propositions \ref{prop-wu-and-faces} (see also the comment after it) and \ref{prop-wu-and-facesT^*}, we get that $\Ext \face (x)\ (\subset B_{X^*})$ is $w^*$-closed, too. This simple remark plays a relevant role in some of the arguments below.\eor
\end{remark}

Lemma \ref{lemma_faces_CK} can be restated as the following corollary.

\begin{corollary}\label{corollary_facesU}
Let $T\colon X\to C(K)$ be an isometric embedding and let $x\in S_X$. A necessary and sufficient condition for $\mu\in S_{C(K)^*}$ to belong to $\face(T(x))$ is that its positive part $\mu^+$ is supported on $F_T^{-1}(\face(x))$ and its negative part $\mu^-$ is supported on $F_T^{-1}(\face(-x))$. 
\end{corollary}

\begin{proof}
The result is straightforward from Lemma~\ref{lemma_faces_CK} and the observations above.
\end{proof}


There is a particular kind of isometric embeddings into $C(K)$-spaces that is of interest. Namely, given a weak$^*$ closed $K\subset B_{X^*}$, we can define the map $u_K\colon X\to C(K)$ by $u_K(x)(k):=\langle x, k \rangle$, for every $x\in X$ and every $k\in K$. The map $u_K$ is an isometric embedding if and only if $\cco^{w^*}(\pm K)=B_{X^*}$. If this is the case, given $x\in S_X$ it holds that
\begin{align*}
\sigma^+(u_K(x))&=F_{u_K}^{-1}(\face(x))=K\cap \face(x),\\
\sigma^-(u_K(x))&=F_{u_K}^{-1}(\face(-x))=K\cap \face(-x).    
\end{align*}
Let us also state, for later reference, the translation of Corollary~\ref{corollary_facesU} for this particular isometric embedding.

\begin{corollary}\label{cor_faces_uE}
Let $X$ be a Banach space and  $K\subset B_{X^*}$ be $w^*$-closed with $\cco^{w^*}(\pm K)=B_{X^*}$. Given $x\in S_X$, a necessary and sufficient condition for $\mu\in S_{C(K)^*}$ to belong to $\face(u_K(x))$ is that its positive part $\mu^+$ is supported on $K\cap \face(x)$ and its negative part $\mu^-$ is supported on $K\cap \face(-x)$. 
\end{corollary}

\section{Some necessary conditions for U- and wU-em\-bed\-da\-bi\-li\-ty in $C(K)$}

Let $T\colon X\to C(K)$ be a wU-embedding
. It was already mentioned above that, as a consequence of Proposition $\ref{prop-wu-and-facesT^*}$,  the restriction of the linear map $T^*$ defines an affine $w^*$-$w^*$-homeomorphism between the faces $\face(Tx)\subset B_{C(K)^*}$ and $\face(x)\subset B_{X^*}$. Thus, we can already conclude that the duals of the spaces $X$ and $C(K)$ must share a similar face structure. Recall that the set of extreme points of the closed unit ball of the dual of a $C(K)$ space is $w^*$-compact; in particular, this is so for the set of extreme points of any of its $w^*$-closed faces. This and the previous observation makes it natural to focus on the following concept:  
\begin{definition}[\cite{Alfsen1971}] Simplices whose proper closed faces have a closed set of extreme points will be called {\bf Bauer simplices}. A {\bf (Bauer) simplexoid} ---see Definition {\em 16.8} in {\em \cite{Phelps2001}}--- is a compact convex set in a locally convex space whose proper closed faces are (respectively, Bauer) simplices. 
\end{definition}

As it was mentioned above, the $w^*$-compact set $B_{C(K)^*}$ is a Bauer simplexoid. Thus, the following necessary condition for wU-embeddability into a  $C(K)$-space holds:

\begin{proposition}\label{prop-wU-simplexoid}
Let $X$ be a Banach space. If $X$ is wU-embeddable  in a space $C(K)$, then $B_{X^*}$ is a Bauer simplexoid. 
\end{proposition}

Next, we present a necessary condition for $T_F$ being a U-embedding.

\begin{lemma}\label{lemma_proper}
Given $F\in C(K,(X^*,w^*))$, a necessary condition for $T_F:X\rightarrow C(K)$ to be a U-embedding is that $F(K)$ is a James boundary and that, given $t,s\in K$  such that $F(t)=F(s)$ (or, alternatively, $F(t)=-F(s)$), then $\|F(t)\|<1$.
\end{lemma}

\begin{proof}
Let $F$ belong to $C(K,(X^*,w^*))$ such that $T_F:X\rightarrow C(K)$ is a  U-embedding. Lemma~\ref{lemma-isometry} implies that the range of $F$ is a James boundary. Assume that there exists $t$ and $s$ in $K$ such that $F(t)=F(s)$ (alternatively, $F(t)=-F(s)$) and $\|F(t)\|=1$. Then (see Remark \ref{rem:delta-extension}),  the functional $x^*=F(t)\in X^*$ admits two different Hahn--Banach extensions through $T_F$, namely $\delta_t$ and $\delta_s$ ($\delta_{t}$ and $-\dd_s$, respectively). This violates the U-embeddability assumption.
\end{proof}

In Lemma~\ref{lemma-isometry} it is established that a necessary and sufficient condition for a function $F$ belonging to $C(K,(X^*,w^*))$ to define an isometry $T_F$ is that $\Ext B_{X^*}$ is covered by $F(K)\cup -F(K)$. 
To refine Lemma~\ref{lemma_proper} a bit more, let us prove the following

\begin{proposition}\label{propo_necessary}
 Let $T\colon X\to C(K)$ be a U-embedding and  $\Az\subset K$ be its \ucore. Then ${{F_T}_{|\Az}\colon  \Az \to F_T(\Az)\subset \Ext B_{X^*}} $ is a homeomorphism (when $B_{X^*}$ is endowed with the $w^*$-topology) and $\|F_T(t)\|<1$ for every $t\in K\setminus \overline{\Az}$. 
\end{proposition}
\begin{proof}

Recall from Proposition \ref{prop-U-homeo} that $T^*:\HB(S_{X^*})\rightarrow S_{X^*}$ is a $w^*$-$w^*$-homeo\-mor\-phism. Since  $\dd(\Az)\subset \HB(S_{X^*})$, this proves the first statement. 

Let $k\in K\setminus\ol \Az$. Assume that $\|F_T(k)\|=1$. Then $\dd_{k}\in\HB(S_{X^*)}$.  Find $f\in C(K)$ with range in $[0,1]$ such that $f(k)=1$ and $f(c)=0$ for each $c\in\ol{\Az}$. To simplify the notation, put $\hbox{Ext}:=\Ext B_{X^*}$ and write $\hbox{Ext}^+$ and $\hbox{Ext}^-$ as in (iv), Subsection \ref{subsec:U-core}. Recall that $\hbox{Ext}=\hbox{Ext}^+\dcup \hbox{Ext}^-$, and that $\hbox{Ext}^-=-\hbox{Ext}^+$.  Find a net $\{x^*_i:\ i\in I,\ \le\}$ in $\co(\hbox{Ext})$ that $w^*$-converges to $F_{T}(k)$. For $i\in I$ put 
\begin{eqnarray*}
\lefteqn{x^*_{i}=\sum_{e^*\in \hbox{Ext}^+}\aa^{i}_{e^*}e^*+\sum_{e^*\in \hbox{Ext}^-}\aa^{i}_{e^*}e^*,\ \hbox{where}}\\
&&\sum_{e^*\in \hbox{Ext}}\aa^{i}_{e^*}=1,\ \aa^{i}_{e^*}\ge 0,\ \hbox{for } e^*\in \hbox{Ext},
\end{eqnarray*}
and the sum above has only a finite number of non-zero summands. Let us define 
\begin{equation*}
    \mu_i:= \sum_{e^*\in \Ext^+} \aa_e^i \dd_{k_{e^*}} + \sum_{e^*\in \hbox{Ext}^-}\aa_e^i(-\dd_{k_{e^*}})
\end{equation*}
where $k_{e^*}$ is the (unique) element in $K$ such that $F_T(k_{e^*})=e^*$ if $e^*\in \hbox{Ext}^+$ and $T^*(-\dd_{k_{e^*}})=e^*$ if $e^*\in \hbox{Ext}^-$. Observe that $\mu_i \in \co(\pm \dd K)\subset B_{M(K)}$ and 
\begin{equation*}
    T^*(\mu_i)= \sum_{e^*\in \Ext^+}\aa_{e^*}^ie^* + \sum_{e^*\in \hbox{Ext}^-}\aa_{e^*}^ie^*= x_i^*, 
\end{equation*}
and also
\begin{equation*}
\begin{split}
    \la f, \mu_i\ra &= \sum_{e^*\in \Ext^+}\aa_{e^*}^i \la f, \dd_{k_{e^*}}\ra + \sum_{e^*\in \hbox{Ext}^-}\aa_{e^*}^i \la f, -\dd_{k_{e^*}}\ra\\
     &= \sum_{e^*\in \hbox{Ext}^+}\aa_{e^*}^i  f(k_{e^*}) + \sum_{e^*\in \hbox{Ext}^-}\aa_{e^*}^i f(k_{e^*})=0
\end{split}
\end{equation*}
for each $i\in I$. If we let $\mu\ (\in B_{M(L)})$ be a $w^*$-cluster point of $\{\mu_i:i\in I,\ \le\}$, then we have $\la f,\mu\ra=0$ and $T^*\mu=F_T(k)$. Since $\|F_T(k)\|=1$ and $\mu \in B_{M(K)}$, it follows that $\mu$ is a Hahn--Banach extension of $F_T(k)$ which is different from $\dd_k$, a contradiction.
\end{proof}

\begin{example}\label{ex:closed-ideals}
\rm
As observed above, in \cite{Phelps1960a}, Phelps showed that given a compact Hausdorff space $K$ and a closed set $Z\subset K$, then its associated closed ideal $I_Z(K)$ has property U in $C(K)$ ---see Proposition~\ref{prop:phelps-ideals}. In this concrete case ---$T$ is just the canonical inclusion $i$--- the homeomorphism ${F_T}_{|\Az}$ in the previous proposition can not be extended in a natural way to ${F_T}_{|\ol{\Az}}$. Indeed, the set $\Az=K\setminus Z$ may not be closed (by just taking $Z$ closed but not open). In this specific U-embedding, we have a homeomorphism between $\Az$ and $F_T(\Az)$ but this is not the case for $\ol{\Az}$ and ${F_T}|_{\ol{\Az}}$ because $F_T^{-1}(0)=Z$ and thus, in particular, for every $t\in \ol{\Az}  \setminus \Az$, $F_T(t)=0$.
\end{example}

We will end this section proving the existence of a very specific type of subset in $B_{X^*}$ that plays a role in the assumption of the U-embeddability of a space $X$ into a space $C(K)$.

\begin{definition} \label{def:U-suitable}
Given a Banach space $X$, a $w^*$-closed set $E\subset B_{X^*}$ will be called {\bf U-suitable} if it satisfies the following properties:
\begin{enumerate}[\rm (i)]
    \item $ E \cap(-E) \cap S_{X^*}=\emptyset$;
    \item $ E \cup (-E) = \overline{\Ext{B_{X^*}}}^{w^*}$. 
\end{enumerate}
If, moreover, it satisfies the following condition:
\begin{enumerate}[\rm (i),resume]
    \item for every $x\in S_X$,  $E \cap \face (x) = E \cap \ol{\Ext \face (x) }^{w^*}$,
\end{enumerate}
we will say that $E$ is a \textbf{proper} U-suitable set.
\end{definition}

Observe that if $E$ is a U-suitable set of $X$, then $\Ext(B_{X^*})\subset (E \cup (-E)) \cap S_{X^*}$, hence $E\cup(-E)$ is, in particular, a James boundary for $X$.

\medskip

The next proposition shows that the existence of a proper U-suitable set is a necessary condition for having a U-embedding into a space $C(K)$. It gives a precise description of the U-suitable set available in such a circumstance, highlighting one of the key properties of the \ucore. Recall that we shortened $\Ext B_{X^*}$ by $\hbox{Ext}$, that $\hbox{Ext}^+$ denotes $\hbox{Ext}\cap F_T(K)$ ($=F_T(\Az)$, where $\Az$ is the U-core, see Definition \ref{def:U-core} above),  and  $\hbox{Ext}^-=-\hbox{Ext}^+$ (see Subsection \ref{subsec:U-core} above).  Observe, too, that $\ol{\hbox{Ext}^+}^{w^*}=F_T(\ol{\Az})$.

\begin{proposition}\label{prop:existence-U-suit}
Let $T\colon X\to C(K)$ be a U-embedding. Then, the set $(F_T(\ol{\Az})=)\ E:= \ol{\hbox{\rm Ext}^+}^{w^*}\subset B_{X^*}$ is a proper U-suitable set.
\end{proposition}

\begin{proof} Notice that $-E=\ol{\hbox{Ext}^-}^{w^*}$. Thus, (i) in Definition \ref{def:U-suitable} is clear.

Since $\hbox{Ext}=\hbox{Ext}^+\dcup\hbox{Ext}^-$, (ii) in the same definition should also be clear.

The technique to prove (iii) is similar to the one used in the proof of Proposition \ref{propo_necessary} above. Notice first that, since $T$ is an wU-embedding, for every $x\in S_X$ the mapping $T^*|_{\face(T(x))}\rightarrow \face(x)$ is a (affine) $w^*$-$w^*$-homeomorphism from the $w^*$-compact set $\face(T(x))$ onto the ($w^*$-compact) subset $\face(x)$. As a consequence (that was mentioned above), the set $\Ext \face(x)$ is $w^*$-closed (see Remark \ref{rem-extreme-in-faces-closed}). Thus, we need to prove that $E \cap \face (x) = E \cap \Ext \face (x)$ for every $x\in S_X$. 

Assume that there exists $e^*_0\in (E \cap \face (x)) \setminus\Ext\face(x)$. There exists a (unique) $k_0\in K$ such that $T^*(\dd_{k_0})=e^*_0$, or, alternatively, $T^*(-\dd_{k_0})=e^*_0$. Assume that the first option holds (the argument for the alternative one is similar). Notice that $\dd_{k_0}\not\in \Ext\face(Tx))$. Thus, there exists $f\in C(K)$ with values in $[0,1]$ such that $f(k)=0$ for $k\in K$ such that $\dd_k\in \Ext\face (Tx)$ and $f(\dd_{k_0})=1$. Since $e^*_0\in \face(x)$, there exists a net $\{x^*_i:\ i\in I, \ \le\}$ in $\co(\Ext \face(x))$ that $w^*$-converges to $e^*_0$. 

Let $\hbox{Ext}^+\face(x):=\Ext\face(x)\cap \hbox{Ext}^+$, and $\hbox{Ext}^-\face(x):=\Ext\face(x)\cap \hbox{Ext}^-$. Put, for $i\in I$,
\begin{eqnarray*}
\lefteqn{x^*_i=\sum_{e^*\in\Ext\face(x)}a^i_{e^*}e^*}\\
&& =\sum_{e^*\in\hbox{Ext}^+\face(x)}a^i_{e^*}e^*+\sum_{e^*\in\hbox{Ext}^-\face(x)}a^i_{e^*}e^*,
\end{eqnarray*}
where $a^i_{e^*}\ge 0$ for all $e^*\in\Ext\face(x)$ and $\sum_{e^*\in\Ext\face(x)}a^i_{e^*}=1$. Given $e^*\in \hbox{Ext}^+\face{x}$, find (the unique)  $k_{e^*}\in K$ such that $T^*(\dd_{k_{e^*}})=e^*$. If, alternatively, $e^*\in \hbox{Ext}^-\face{x}$, find (the unique) $k_{e^*}\in K$ such that $T^*(-\dd_{k_{e^*}})=e^*$. 
Let us define, for $i\in I$,
$$
\mu_i:=\sum_{e^*\in\hbox{Ext}^+\face(x)}a^i_{e^*}\dd_{k_{e^*}}+\sum_{e^*\in\hbox{Ext}^-\face(x)}a^i_{e^*}(-\dd_{k_{e^*}}).
$$
Notice that $\mu_i\in \co(\pm \dd K)\ (\subset B_{M(K)})$ and, obviously, $T^*\mu_i=x^*_i$, for $i\in I$. The net $\{\mu_i:\ i\in I,\ \le\}$ has a $w^*$-cluster point $\mu_0\in \cco^{w^*}(\pm\dd K)\ (\subset B_{M(K)})$, and clearly $T^*(\mu_0)=e^*_0$. Thus, $\mu_0$ is a Hahn-Banach extension of $e^*_0$. Notice that $\langle f,\mu_i\rangle=0$ for all $i\in I$, hence $\langle f,\mu_0\rangle=0$. Since $\langle f,\dd_{k_0}\rangle=1$, we got two different Hahn--Banach extensions of $e^*_0$, a contradiction.
\end{proof}

\begin{remark}
\rm Observe that for (iii) in Definition \ref{def:U-suitable} for the set $E$ in Proposition \ref{prop:existence-U-suit}, wU-embeddability of $X$ into $C(K)$ suffices.
\end{remark}

\section{Spaces that cannot be U-embedded in $C(K)$}

We have obtained in the previous section the key necessary condition for a Banach space $X$ to be U-embedded into a $C(K)$ space; namely, the existence of a U-suitable set. This allows us to provide some classes of spaces that are not U-embeddable into a space $C(K)$. 

In that direction, the feeling that whenever $X$ is U-embedded into a space  $C(K)$, then $B_{X^*}$ can not have too many extreme points is supported by the following result, a consequence of Proposition \ref{prop:existence-U-suit} above: 


\begin{corollary}\label{cor-connection}
Let $X$ be a Banach space such that $(\ol{\Ext B_{X^*}}^{w^*}\cap S_{X^*},w^*)$ is connected. Then $X$ cannot be U-embedded into any $C(K)$-space. 
\end{corollary}
\begin{proof} Assume that $X$ is U-embedded in a space $C(K)$. Proposition \ref{prop:existence-U-suit} and (i) in Definition \ref{def:U-suitable} show that $(\ol{\Ext B_{X^*}}^{w^*}\cap S_{X^*},w^*)$ is disconnected.
\end{proof}

\begin{corollary}\label{cor-gateaux}
No G\^ateaux smooth Banach space can be U-embedded into a $C(K)$ space.
\end{corollary}
\begin{proof}
If $X$ is a G\^ateaux smooth Banach space then $\NA(X)\cap S_{X^*}\subset \Ext B_{X^*}$. By the Bishop--Phelps theorem,  $\NA(X)\cap S_{X^*}$ is norm-dense in $S_{X^*}$ and, therefore $\ol{\Ext B_{X^*}}^{w^*}\cap S_{X^*}=S_{X^*}$, which is obviously connected in the $w^*$-topology. The result follows from  Corollary \ref{cor-connection}. 
\end{proof}

We will see later ---see Corollary~\ref{coro-char-2dim}--- that for two-dimensional subspaces, being non-G\^ateaux smooth is equivalent to U-embeddability into $C(K)$-spaces.

\section{Construction of U-embeddings}

\subsection{A sufficient condition in terms of Bauer simplexoids}
Let us first show ---as a consequence of Bauer characterization of extreme points of compact convex subsets of locally convex spaces--- that a U-suitable set $E$ give uniqueness of the extension for extreme points of the unit ball of $X^*$ that are simultaneously norm-attaining, through its associated isometric embedding $u_E$.

\begin{proposition}\label{prop-uE} Let $X$ be a Banach space for which there exists a U-suitable set $E$. Then the map $u_E:X\rightarrow C(E)$ is an isometric embedding such that for every $x^*\in\Ext{B_{X^*}}\cap\NA{(X)}$ there exists a unique norm-one measure $\mu$ on $E$ such that $u_E^*(\mu)=x^*$. Precisely, if $x^*\in E$ then $\mu=\dd_{x^*}$, and if $x^*\in -E$ then $\mu=-\dd_{-x^*}$.  
\end{proposition}

\begin{proof}
That $u_E:X\rightarrow C(E)$ is an isometric embedding was already noticed above (see the paragraph after Definition \ref{def:U-suitable}). Fix $x^*\in\Ext{B_{X^*}}\cap\NA{(X)}$ that attains its norm at, say, $x\in S_X$. Assume that $\nu$ in $S_{M(E)}$ is such that $u_E^*(\nu)=x^*$. By Corollary~\ref{cor_faces_uE}, it holds that $\nu^+$ is supported on $E\cap \face(x)$ and $\nu^-$ is supported on $E\cap \face(-x)$. Define the measure $\mu$ on $B_{X^*}$ by 
\[
\mu(A):=\nu^+(A\cap E\cap \face(x))+\nu^-(-A\cap E\cap \face(-x)),
\] 
for every Borel subset $A$ of $B_{X^*}$. It is clear that $\mu$ is a probability measure on $B_{X^*}$ ---it is well defined due to the fact that $\face(x)$ and $\face(-x)$ are $w^*$-closed---. It is also easy to see that for every $z\in X$, it holds that
\begin{align*}
\int z\,d\mu&= \int_{E\cap \face(x)} z\,d\mu+\int_{-E\cap \face(x)} z\,d\mu=\\
&=\int_{E\cap \face(x)} z\,d\nu^+-\int_{E\cap \face(-x)} z\,d\nu^-=\\
&=\int z\,d\nu= \langle u_E(z),\nu\rangle=\langle z,x^*\rangle.    
\end{align*}
By Bauer theorem~\cite[Proposition 1.4]{Phelps2001} applied to the convex and $w^*$-compact set $B_{X^*}$, we deduce that $\mu=\delta_{x^*}$, which implies that $\nu=\delta_{x^*}$ in case $x^*$ belongs to $E$, and that $\nu=-\delta_{-x^*}$ in case $x^*$ belongs to $-E$.
\end{proof}

\begin{definition}[{See Propositions 1.1 and 1.2 in \cite{Phelps2001}}]\label{def-resultant}
Let $K\subset S$ be a compact convex subset of the locally convex space $(S,\mathcal T)$. There exists and affine, onto,  and $w^*$-$\mathcal T$-continuous map $r_K\colon P(\ol{\Ext(K)})\to K$ (where $P(\ol{\Ext(K)})$ is the set of all  probability measures on $\ol{\Ext(K)}$) such that, for every $f\in S^*$ and every $\mu\in  P(\ol{\Ext(K)})$,
\[
\int f\,d\mu = \langle r_K(\mu),f\rangle.
\]
The map $r_K$ is called the {\bf resultant map}. The point $r_K(\mu)\ (\in K)$ is usually called the {\bf resultant} ---or {\bf barycenter}--- of the measure $\mu$. 
\end{definition}

In the concrete case where $X$ is a Banach space and $S=(X^*,w^*)$, since $S^*=X$, it follows that for every $x^*\in K$ the preimage $(r_K)^{-1}(x^*)$ is precisely the set of representing probability measures of $x^*$ on $K$; that is, a probability measure $\mu$ belongs to $(r_K)^{-1}(x^*)$ if, and only if, it satisfies $\langle y, x^*\rangle =\int_K y\,d\mu$ for every $y\in X$ ---in a similar way, though not completely analogous, to the Riesz representation theorem in the case $X = C(K)$.

\medskip

Let $X$ be a Banach space. Given $x\in S_X$, put
$$
\ex(x):=\overline{\Ext{\face(x)}}^{w^*}\ \Big(\subset \face(x)\subset S_{X^*}\Big).
$$
Certainly, $\face (x)=\cco^{w^*}(\ex(x))$. 

\medskip

Put 
$$
\ex^+(x):=\ex(x)\cap E,\ \hbox{and } \ex^-(x):=\ex(x)\cap -E.
$$
Thus, 
\begin{equation}\label{eq-E(x)}
\ex(x)=\ex^+(x)\ \dcup \ \ex^-(x),
\end{equation}
and both sets, $\ex^+(x)$ and $\ex^-(x)$, are $w^*$-closed.

The following lemma collects some easy facta for later reference. Its proof shall be omitted. 

\begin{lemma}\label{lemma-U-suitable}
Let $E$ be a U-suitable set for a Banach space $X$, and let $u_E:X\rightarrow C(E)$ be the natural isometric embedding defined above. Let $x\in S_X$. Then
\begin{enumerate}[\rm (i)]
    \item $\Ext \face\big(u_{E}(x)\big)=\dd\big(E\cap\face(x)\big)\dcup -\dd\big(E\cap\face(-x)\big)$.
    \item $\Ext\face(x)\subset \ex(x)\subset \big(E\cap\face(x)\big)\dcup \big(-E\cap\face(x)\big)$.
    \end{enumerate}
    If $E$ is, moreover, proper, then
\begin{enumerate}[\rm (i)]\setcounter{enumi}{2}
\item $\ex(x)= \big(E\cap\face(x)\big)\dcup \big(-E\cap\face(x)\big)$. More precisely, $\ex^+(x)=E\cap\face(x)$ and $\ex^-(x)=-E\cap\face(x)$. 
\end{enumerate}
\end{lemma}

We continue now the work done in Subsection \ref{subsec:face-structure}.

\medskip

Let $X$ be a Banach space. Fix $x\in S_X$. In order to simplify the notation in the following two results, we shall write 
$$
\face:=\face(x)\ (\subset B_{X^*})
$$ 
and 
$$
\face_{u}:=\face\big(u_{E}(x)\big)\ (\subset B_{M(E)}).
$$

\begin{lemma}
Let $X$ be a Banach space. Let us assume that $E\subset B_{X^*}$ is a U-suitable set for $X$. Let $x\in S_X$. Then, the resultant mapping ${r_{\mathfrak{F}_u}:P(\Ext\face_u)\rightarrow \face_u}$ is one-to-one and onto.
\end{lemma}

\begin{proof}
Let us prove first injectivity. It is enough to show that  
\begin{equation}\label{eq-r-and-mu}
r_{\face_u}(\mu)(B)=\mu\big(\dd(B\cap\face)\big)-\mu\Big(-\dd\big(B\cap\face(-x)\big)\Big)
\end{equation}
for every compact subset $B$ of $E$ and every $\mu\in P(\Ext\face_u)$. Let us sketch the proof.  The reader may easily fill the $\ee$'s missing. The set $\face_u$ is a $w^*$-compact convex subset of the locally convex space $(M(E),w^*)$. Thus, given $f\in C(E)$, the definition of resultant shows that 
\begin{equation}\label{eq-barycenter}
\int f\, d\mu=\int_{\Ext\face_u}f\, d\mu=\langle f,r_{\face_u}(\mu)\rangle.
\end{equation}
Find $g\in C(E)$ such that $g(e^*)=1$ for $e^*\in B$ and that is almost $0$ on $E\setminus B$. Notice that the way $g$ acts on $M(E)$ shows that   $\langle g,-\dd_{e^*}\rangle=-g(e^*)$ for all $e^*\in E$. 
Then
\begin{eqnarray*}
\lefteqn{r_{\face_u}(\mu)(B)\approx \langle g,r_{\face_u}(\mu)\rangle}\\
&&=\int_{\Ext\face_u}g\, d\mu=\int_{\dd(E\cap\face)}g\, d\mu+\int_{-\dd(E\cap\face(-x))}g\, d\mu\\
&&\approx \mu\big(\dd(B\cap\face)\big)-\mu\Big(-\dd\big(B\cap\face(-x)\big)\Big),
\end{eqnarray*}
where the first equality comes from (\ref{eq-barycenter}). This (almost) shows (\ref{eq-r-and-mu}) and then the sought injectivity. 

Surjectivity is a consequence of a general representation theorem in this case, since $\Ext\face_u$ is $w^*$-closed.\end{proof}

\begin{lemma} \label{lemma-measures-measures}
Let $X$ be a Banach space.  Let us assume that $E\subset B_{X^*}$ is a U-suitable set for $X$. Then, for every $x\in S_X$ there is a one-to-one and $w^*$-$w^*$-continuous mapping $P:P\big(\ex(x)\big)\to P(\Ext \face_u)$. If $E$ is, moreover, proper, then $P$ is onto. 
\end{lemma}
\begin{proof} 
Let $\mu\in P\big(\ex(x)\big)$. We shall define a regular Borel measure $\wh \mu:=P(\mu)$ on $\Ext \face_u$. It is enough to define $\wh\mu$ on the family of all closed (i.e.,  compact) subsets $\wh A$ of $\Ext \face_u$. For such a set $\wh A$, put
\begin{equation}
\wh \mu(\wh A):=\mu \big(u_E^*(\wh A)\cap \ex(x)\big)
\end{equation}
It is well defined, since $u^*_E(\wh A)$ is a $w^*$-compact set. 
We claim that $\wh\mu$ is a probability measure on $\Ext \face_u$. Indeed, $\wh\mu$ is non-negative, and
\begin{equation}
\wh\mu(\Ext \face_u)=\mu\big(\ex(x)\big)=1.
\end{equation}
Let us prove that $P$ is $w^*$-$w^*$-continuous. To this end, observe that the mapping $u_E^*:M(E)\rightarrow X^*$ is $w^*$-$w^*$-continuous, and $u_{E}^*|_{\hbox{\small{Ext}\ }\face_u}$ maps the $w^*$-compact set $\Ext\face_u$ onto the $w^*$-compact set $(\pm E)\cap\face$ in a one-to-one way. Precisely, $u_{E}^*(\dd_{e^*})=e^*$ for $e^*\in E\cap\face$, and $u_{E}^*(-\dd_{e^*})=-e^*$ for $e^*\in E\cap\face(-x)$. Its inverse, denoted by $\Delta:(\pm E)\cap \face\rightarrow \Ext\face_u$, is hence also $w^*$-$w^*$-continuous.

Let $\{\mu_i\}$ be a net in $P(\ex(x))$ that $w^*$-converges to $\mu\in P(\ex(x))$. Let $f\in C\Big(\Ext \face\big(u_{E}(x)\big)\Big)$. Put $g:=f\circ\Delta$. It follows that $g\in C\big(\ex(x)\big)$. Thus, $\langle g,\mu_i\rangle\rightarrow \langle g,\mu\rangle$, hence $\langle f,P(\mu_i)\rangle\rightarrow \langle f,P(\mu)\rangle$, as we wanted to show. From (ii) in Lemma \ref{lemma-U-suitable} it follows easily that $P$ is one-to-one. If $E$ is, moreover, proper, the surjectivity of $P$ follows from the fact (see (i) and (iii) in Lemma \ref{lemma-U-suitable}) that $\ex(x)$ and $\Ext\face_u$ are homeomorphic in their $w^*$-topologies.\end{proof}

The following result provides a sufficient condition for the existence of a wU-embedding. Observe that the Bauer-simplexoid condition on $B_{X^*}$ is unavoidable: The existence of a wU-embedding implies that the $w^*$-closed faces of $B_{X^*}$ are affinely homeomorphic to $w^*$-closed faces of some $C(K)$ space, and so the set ot its extreme points should be $w^*$-closed. 

\begin{theorem}\label{theoremsimplexoid-w-U}
Let $X$ be a Banach space admitting a proper U-suitable set $E$ and whose dual ball is a Bauer simplexoid. Then the map $u_E$ is a wU-embedding from $X$ into $C(E)$.
\end{theorem}

\begin{proof}
Let us assume that $B_{X^*}$ is a Bauer simplexoid and that $E$ is a U-suitable set for $X$. It is clear that $u_E$ is an isometric embedding. We claim that $u_E$ is a wU-embedding. Indeed, fix $x\in S_X$. By Lemma \ref{lemma-measures-measures} there exits a bijective function $\gamma\colon \face(u_E(x))\to P(\ex{(x)})$ such that $u_E^*=r_x\circ \gamma$. If a face is a Bauer simplex, then the probabilities in $P(\ex(x))$ are maximal (incidentally, this is an equivalence), since they are supported on the extreme points of the face. Then, Choquet--Meyer uniqueness theorem implies that the resultant is bijective from the maximal probability measures onto  $\face(x)$. 
\end{proof}

\subsection{Finite-dimensional spaces}

Along this section $X$ will be a finite-dimensional Banach space. We are interested in characterizing when it can be U-embedded into a $C(K)$-space. Observe that in this case ---as finite-dimensional spaces are reflexive--- U-embeddings and wU-embeddings coincide. Also, notice that any simplexoid in a finite dimensional space is a Bauer simplexoid.

By Proposition~\ref{prop-wu-and-facesT^*}, a necessary condition on $X$ to be U-embedded into some $C(K)$-space is that its dual faces $\face(x)$ are Bauer simplices. Indeed, \cite[Theorem 3.2]{Phelps1960a} implies that given $x\in S_X$, there exist a set $\{t_i\}_{i=1}^k\subset K$ and signs $\{\varepsilon_i\}_{i=1}^k$, with $k\leq n$ such that $\face(x)$ is afine homeomorphic to the convex hull of the afinelly independent set  $\{\varepsilon_i\delta_{t_i}^K\}_{i=1}^k\ (\subset S_{M(K)})$.


\begin{theorem}\label{thm:finite-Uembedd}
A finite-dimensional Banach space $X$ can be U-embedded into a $C(K)$-space if, and only if, it admits a proper U-suitable set and its dual unit ball is a simplexoid.
\end{theorem}
\begin{proof}
One direction follows from Theorem \ref{theoremsimplexoid-w-U}. The other follows from Propositions \ref{prop-wU-simplexoid} and \ref{prop:existence-U-suit}.
\end{proof}

Since any ball in $\mathbb{R}^2$ is a simplexoid and every U-suitable set is proper ---the set of extreme points of $B_{X^*}$ is closed--- we deduce the following corollary.

\begin{corollary}\label{cor-2-U-suitable}
A two-dimensional Banach space $X$ can be U-embedded into a $C(K)$-space if and only if it admits a U-suitable set.
\end{corollary}

Again a concrete geometric property of two-dimensional spaces gives the following characterization.

\begin{corollary}\label{coro-char-2dim}
A two-dimensional Banach space $X$ can be U-embedded into a $C(K)$-space if and only if it is not G\^ateaux smooth.
\end{corollary}
\begin{proof}
On one hand, we now that a Banach space that can be U-embedded into a $C(K)$-space can not be G\^ateaux. On the other hand, assume that $X$ is not G\^ateaux and pick $x\in S_X$ where the norm is not G\^ateaux smooth. Let $x^*$ and $y^*$ be the two extreme points of $\face(x)$ and set $z^*=(x^*+y^*)/2$. Take now $z\in S_X$ such that $z^*(z)=0$. Then, it is easy to see that $E=\{e^*\in\Ext B_{X^*}: e^*(z)>0\}$ is a U-suitable set for $X$. The conclusion follows from Corollary \ref{cor-2-U-suitable}
\end{proof}

Thus, a consequence of Proposition \ref{prop:retraction} below and the argument in the proof of Corollary \ref{coro-char-2dim}, the space $C[0,1]$ is ``U-embedding''-universal for the $2$-dimensional non-G\^ateaux spaces. Indeed, every such a space has a U-suitable set that is simultaneously a retract and a zero-set of the interval $[0,1]$. We can then use Proposition \ref{prop:retraction} below.

\begin{remark}
\rm
The corollaries above are obtained as a consequence of some previous theorems. It is worth to notice that, for finite-dimensional spaces, it is not difficult to get them through explicit computations. For example, if we assume $X$ being a finite-dimensional polyhedral space with simplexoid dual ball, then we also know that $B_{X^*}$ has an even number of finite extreme points, say $2n$. Thus, one can just take $E\subset \Ext B_{X^*} $ a selection of $n$ ``proper'' extreme points (namely, for each $x^*\in \Ext B_{X^*}$, we can just take either $x^*$ or $-x^*$). Trivially, such set $E$ satisfies the conditions of a U-suitable set, and thus $X$ can be embedded in $C(E) \equiv C(\{1,2,...,n\})$ through the operator $u_E$. Also, even though it is obtained as a particular consequence of Theorem \ref{thm:finite-Uembedd}, it is easy to explicitly prove that every $x^*\in X^*$ has a unique Hahn--Banach extension through $u_E$. \eor
\end{remark}

\begin{remark}
\rm

At first glance we may think that, at least in finite-dimensional spaces, the notion of U-embeddability might have something to do with the one of polyhedrality. This is not so, and the example below illustrates the gap between both concepts.

\medskip

{\bf Example}: Let $\{\theta_n\}_n\subset[0,\pi/2)$ be an increasing sequence such that $\theta_1=0$ and $\lim_n \theta_n=\pi/2$. Set the following vectors of $\mathbb{R}^2$: $x_n^*:=(\cos(\theta_n),\sin(\theta_n))$ and $x_{\bullet}^*:=(0,1)$. Define $X$ as $\mathbb{R}^2$ endowed with the norm
\[
\|(x,y)\|=\max\{\sup_n\{|\langle(x,y),x^*_n\rangle|\},|\langle(x,y),x^*_{\bullet}\rangle|\}.
\]
By definition, the set  $E=\{\pm x_n^*\colon n\in\mathbb{N}\}\cup\{\pm x_{\bullet}^*\}$ is a $1$-norming subset of $B_{X^*}$. It is additionally closed so ---by Milman's converse of the Krein-Milman theorem--- $E$ is the set of extreme points of $B_{X^*}$, and $B_{X^*}$ is the convex hull of $E$.\\
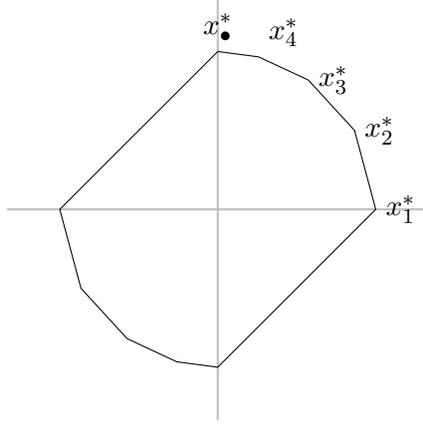
\centerline{\begin{tikzpicture}[scale=0.7]
\draw[thick, lightgray] (-4,0) --(4,0);
\draw[thick, lightgray] (0,-4) --(0,4);
\draw(0:3) node[right]{$x_1^*$}--(30:3)node[right]{$x_2^*$} --(55:3) node[right]{$x_3^*$}--(75:3)node[anchor= south west]{$x_4^*$}--(90:3) node[above]{$x_\bullet^*$}--(-3,0)
-- (210:3)--(235:3)--(255:3)--(270:3)--(3,0);
\end{tikzpicture}}
Let us consider the compact set $K=\{x_n^*\colon n\in\mathbb{N}\}\cup\{x_{\bullet}^*\}\subset B_{X^*}$. Set $F\colon K\to X^*$ by $F(t):=t$ for every $t\in K$. This function defines an isometric embedding $T_F\colon X\to C(K)$ given by
\[
T_F(x)(t):=\langle x, F(t)\rangle
\]
for every $x\in X$ and every $t\in K$. It is easy to see that $T_F$ is indeed a U-embedding by Corollary \ref{coro-char-2dim}, as $X$ is clearly non-G\^ateaux differentiable. In particular, $K$ is indeed a U-suitable set. Hence, this is an example of a Banach space $X$ which can be U-embedded in $C(K)$, $B_{X^*}$ is a simplexoid, but it is not a polyhedral space, as it has infinitely many faces.

\medskip

Corollary \ref{cor:c-beta-non-poly} below gives an example of an  infinite-dimensional space which is also non-polyhedral although it is U-embedded in $C(K)$.\eor

\end{remark}

\subsection{Embeddings between $C(K)$-spaces}

Through this subsection we characterize the pairs of compact Hausdorff spaces $(K,S)$ for which $C(K)$ can be U-embedded into $C(S)$. The following result shows that the natural isometric embeddings between $C(K)$-spaces are never U-embeddings.

\begin{lemma}\label{lemma_compositionOperator}
Let $h\colon S\to K$ be a continuous onto map between the compact Hausdorff spaces $S$ and $K$. The isometric embedding $C_h\colon C(K)\to C(S)$ given by $C_h(f)=f\circ h$ for $h\in C(K)$ is a U-embedding if, and only if, $h$ is one-to-one ---so $S$ and $K$ are homeomorphic---, and if, and only if, $C_h$ is onto.
\end{lemma}
\begin{proof}
It is easy to see that the isometry $C_h$ coincides with $T_F$ for $F(s):=\delta_{h(s)}$, $s\in S$ (see equations (\ref{eq-varphi}) and (\ref{eq-phi}) above). If $h$ is not one-to-one then there exist $s_1$ and $s_2$ in $S$ such that $s_1\not=s_2$ and $h(s_1)=h(s_2)$. Thus, $F(s_1)=F(s_2)\in S_{M(K)}$. An appeal to Lemma~\ref{lemma_proper} gives that $T_F\ (=C_h)$ is not a U-embedding.

Notice that $C_h$ is always one-to-one. Thus, $C_h$ is onto if, and only if, it is an isometry from $C(K)$ onto $C(S)$. This happens if, and only if, $S$ and $K$ are homeomorphic.
\end{proof}

Let us recall the following result from \cite{Aviles2014}. We need first a definition:

\begin{definition}\label{def-u-v}
Let $S$ be a Hausdorff compact topological space. Let $f\in C(S)$ and $0<\sigma<\ee$. Let  
$u^{f}_{\sigma,\ee}\in B_{C(S)}$ such that 
$$
u^{f}_{\sigma,\ee}|_{\{f\ge 1-\sigma\}}\equiv 1\ \ \hbox{and }\  u^{f}_{\sigma,\ee}|_{\{f\le 1-\ee\}}\equiv 0,
$$
and let 
$v^{f}_{\sigma,\ee}\in B_{C(S)}$ such that 
$$
v^{f}_{\sigma,\ee}|_{\{f\le -1+\sigma\}}\equiv 1\ \ \hbox{and }\  v^{f}_{\sigma,\ee}|_{\{f\ge -1+\ee\}}\equiv 0.
$$
Given $\mu\in M(S)$, let us define $\mu^{f,1}_{\sigma,\ee}$ and $\mu^{f,2}_{\sigma,\ee}$ in $M(S)$ by
$$
\mu^{f,1}_{\sigma,\ee}(g):=\int_{S}g.u^{f,1}_{\sigma,\ee}\, d \mu\ \ \hbox{and }\ \mu^{f,2}_{\sigma,\ee}(g):=\int_{S}g.v^{f,1}_{\sigma,\ee}\, d \mu\ \ \hbox{for all }g\in C(S).
$$
\end{definition}

\begin{lemma}\label{lemma_almostattaining}
Let $S$ be a Hausdorff compact topological space. Let $f\in B_{C(S)}$ and $\mu\in B_{M(S)}$. Then, for every $0<\sigma<\ee<1$, it holds that

{\em (i)} $\|\mu^{f,1}_{\sigma,\ee}\|\le 1$ and $\|\mu^{f,1}_{\sigma,\ee}\|\le 1$.

{\em (ii)} $\|(\mu^{f,1}_{\sigma,\ee})^+\|+\|(\mu^{f,2}_{\sigma,\ee})^-\|\ge 1-(1-\langle f,\mu\rangle)/\sigma$.

{\em (iii)}  $\|(\mu^{f,1}_{\sigma,\ee})^-\|+\|(\mu^{f,2}_{\sigma,\ee})^+\|\le (1-\langle f,\mu\rangle)/\sigma$.

{\em (iv)}  $\|\mu-\mu^{f,1}_{\sigma,\ee}-\mu^{f,2}_{\sigma,\ee}\|\le (1-\langle f,\mu\rangle)/\sigma$.
\end{lemma}
The following will be a useful characterization.

\begin{theorem}\label{theorem_cks}
Let $K$ and $S$ be compact Hausdorff spaces. An isometric embedding $T\colon C(K)\to C(S)$ is a U-embedding if, and only if, there exist a closed subset $S_0$ of $S$, an homeomorphism $h\colon K\to S_0$, and a continuous function $\varepsilon\colon K\to \{\pm 1\}$ such that

{\em (i)} $F_T\big(h(k)  \big)=\varepsilon(k)\delta^K_k$ for every $k\in K$, and 

{\em (ii)} $\|F_T(s)\|<1$ for every $s\in S\setminus S_0$. 
\end{theorem} 

\begin{proof}
Let us assume first that $T$ is a U-embedding. Put 
\begin{equation}\label{eq-S0}
S_0:=F_T^{-1}\big(\dd(K)\big)\ (\subset S).
\end{equation}
Obviously, $S_0$ is closed. Proposition \ref{prop:U_homeoT^*} above shows that $F_T|_{S_0}:S_0\rightarrow \dd(K)$ is a $w^*$-$w^*$-homeomorphism from $S_0$ onto $\dd(K)$. Put $h:=(F_T|_{S_0})^{-1}\circ\dd^K:K\rightarrow S_0$. It is a $w^*$-$w^*$-homeomorphism from $K$ onto $S_0$. The equation $T^*(\dd_{h(k)})=\ee(k)\dd^{K}_{k}$ defines a mapping $\ee=K\rightarrow \{\pm 1\}$. It is easy to show that $\ee$ is continuous. 
Finally, the fact that $\|F_T(s)\|<1$ for every $s\in S\setminus S_0$ follows from Proposition \ref{propo_necessary}.

Let us now show a sketch of a proof for the reciprocal. Assume $T\colon C(K)\to C(S)$ is an isometric embedding with the properties in the statement.  

Pick $\mu\in S_{M(K)}$ and take $\nu_1$ and $\nu_2$ in $S_{M(S)}$ such that $T^*(\nu_1)=T^*(\nu_2)=\mu$. Choose a sequence of functions $\{f_n\}_n\subset S_{C(K)}$ such that $\{\langle f_n,\mu\rangle\}_n$ is increasingly convergent to $1$. This, in particular, means that the sequence $\{\langle T(f_n),\nu_i\rangle\}_n$ also increasingly converges to $1$ for $i=1,2$. Let $Z$ be a compact subset of $S\setminus h(K)$ such that $\nu_i\Big(\big(S\setminus h(K)\big)\setminus Z\Big)$ is small for $i=1,2$. Since $\|F_T(s)\|<1$ for every $s\in S\setminus h(K)$, we have $\sup\{\|F_T(s)\|:\ s\in Z\}<1$. Then we can find $n_0\in\N$ such that $\sigma^+\big(T(f_n)\big)\cup \sigma^-\big(T(f_n)\big)\subset S\setminus Z$ for every $n\ge n_0$. An easy consequence of Lemma~\ref{lemma_almostattaining} shows then that the measures $\nu_i$, $i=1,2$, are supported on $h(K)$. Given any function $f\in S_{C(S)}$, consider the function $g= \varepsilon\cdot (f\circ h)$, that belongs to $B_{C(K)}$. Then, for $s\in h(K)$ it holds that there exists $t\in K$ with $s=h(t)$ and
\[
T(g)(s)=\langle g, F_T\big(h(t)\big)\rangle= \langle g,\varepsilon(t)\delta^K_t\rangle=\varepsilon(t)g(t)=f(s).
\]
Since the support of $\nu_i$ is contained in $h(K)$, it follows that
\[
\langle f, \nu_i\rangle=\langle T(g), \nu_i\rangle= \langle g, T^*(\nu_i)\rangle=\langle g, \mu\rangle,\]
which, in particular, implies that $\nu_1=\nu_2$ as we wanted to show.
\end{proof}

\begin{remark}
\label{rem-u-core}
\rm Observe that in case that a linear isometry $T:C(K)\rightarrow C(S)$ is a U-embedding, then the set $S_0$ defined in (\ref{eq-S0}) is the U-core $\Az$ of $T$ (see Definition \ref{def:U-core} above). Indeed, given $k\in K$ there is a single $s_0\in S_0$ such that either $T^*(\dd_{s_0})=\dd_k$ or $T^*(-\dd_{s_0})=\dd_k$. This element is precisely $h(k)$. No element in $\dd(S\setminus S_0)$ can be applied to some $\dd_k$ for $k\in K$, since $\|T^*(\dd_{s})\|<1$.\eor  
\end{remark}

The following example is a consequence of Theorem \ref{theorem_cks}. It is interesting due to the fact that the space of convergent sequences $c$ has property $\beta$ but it is not polyhedral.

\begin{corollary}\label{cor:c-beta-non-poly}
The Banach space $c$ can be U-embedded into $C[0,1]$.
\end{corollary}

\begin{proof}
Define $K$ as the one-point compactification of $\mathbb{N}$. Then $c$ is isometrically isomorphic to $C(K)$. Set $F\colon [0,1] \to B_{M(K)}$ by $F(0)=\delta_\infty$, $F(1/n)=\delta_n$, and for $t\in[1/(n+1),1/n]$ the point $F(t)$ runs along the Bezier curve with control points $\delta_{n+1}$, $0$ and $\delta_{n}$ starting at $\delta_{n+1}$ and ending at $\dd_n$, for $\nin$.  The function $F$ is $w^*$-continuous, and defines an isometric embedding $T_F$ as in equation  (\ref{eq-phi}) above. Since $\|F(t)\|<1$ whenever $t\in[0,1]$ with $t\neq 1/n$ or $t\neq 0$, the function $h:K\rightarrow [0,1]$ defined as $h(n)=1/n$ for $\nin$ and $h(\infty)=0$ is a homeomorphism from $K$ into $S_0:=\{0,1,1/2,1/3,\ldots\}\ (\subset [0,1])$ that satisfies the set of sufficient conditions in Theorem \ref{theorem_cks}. This shows the statement. \end{proof}

As mentioned in Proposition \ref{prop:phelps-ideals}, Phelps \cite{Phelps1960a}  showed that given a closed subset $A$ of a compact Hausdorff space $K$, the ideal $I_K(A)\subset C(K)$, consisting of those continuous function vanishing at every point of $A$, has property U in $C(K)$. As a consequence of this and the transitivity of the U-embeddings (see Lemma \ref{lemma-chain}) we have the following corollary.

\begin{corollary}\label{cor-c0}
The Banach space $c_0$ can be U-embedded into $C[0,1]$.
\end{corollary}

\begin{remark}\label{rem-renorming}
\rm Corollary \ref{cor-c0} shows something that should be clear from the very beginning: The property of being U-embedded into a Banach space is not stable by renormings. Indeed, $c_0$ has an equivalent G\^ateaux smooth norm $\nn_1$, as it does every separable Banach space. However, $(c_0,\nn_1)$ cannot be embedded in any $C(K)$ space (see Corollary \ref{cor-gateaux} above).\eor
\end{remark}

A topological subspace $Z$ of a compact space $S$ is called a {\bf zero-set} whenever there exists $f\in C(S)$ such that $Z=\{t\in S:\  f(t)=0\}$. The next lemma is well known; we omit the proof.

\begin{lemma} \label{lema_same_Gd-zs}
Let $Z$ be a compact subspace of a normal space $S$. Then the following conditions are equivalent:
\begin{enumerate}[\rm (i)]
    \item $Z$ is a zero-set in $S$.
    \item $Z$ is a $G_\delta$-set in $S$.
\end{enumerate}
\end{lemma}

\begin{corollary}\label{coro_zerosets}
Let $S$ and $K$ be two compact Hausdorff spaces. Assume that there exists a U-embedding $T$ from $C(K)$ into $C(S)$. Then the homeomorphic copy of $K$ inside $S$ given by Theorem~{\em{\ref{theorem_cks}}} (i.e., the \ucore\ $\Az$, see Remark {\em \ref{rem-u-core}}) is a zero-set in $S$ and so a $G_\delta$ set in $S$.
\end{corollary}

\begin{proof}
Let $h:K\rightarrow h(K)=:S_0\ (\subset S)$ and $\ee:K\rightarrow \{\pm 1\}$ be the functions given by Theorem \ref{theorem_cks}. Observe that
$$
\langle T\ee,\dd_{s_0}\rangle=\langle \ee,T^*\dd_{h(k)}\rangle=\langle \ee,\ee(k)\dd_k\rangle=(\ee(k))^2=1,\ \hbox{ for }s_0=h(k)\in S_0,
$$
while
$$
|\langle T\ee,\dd_{s}\rangle|=|\langle \ee,T^*\dd_s\rangle|\le\|T^*\dd_s\|<1,\ \hbox{ for }s\in S\setminus S_0.
$$
This shows the statement. \end{proof}

Following \cite{Pelczynski1968} and given a continuous embedding $h\colon K\to S$ between compact spaces, a linear operator $T\colon C(K)\to C(S)$ is said to be a {\bf linear extension operator} for $h$  whenever $F_T\circ h =\delta^K$. This terminology allows us to restate Theorem~\ref{theorem_cks} as follows.

\begin{theorem}\label{theorem_equiv-with-eos}
Let $K$ and $S$ be compact spaces. The following conditions are equivalent:
\begin{enumerate}[\rm (i)]
    \item There exists a U-embedding $T\colon C(K)\to C(S)$.
    \item There exists a continuous embedding $h\colon K\to S$ admitting a norm-one extension operator with $h(K)$ a $G_\delta$-set of $S$.
\end{enumerate}
\end{theorem} 

\begin{proof}
(i)$\Rightarrow$(ii). Assume that $T\colon C(K)\to C(S)$ is a U-embedding. By Theorem~\ref{theorem_cks}, there exists an homeomorphic embedding $h\colon K\to S$ and a continuous map $\varepsilon\colon K\to\{\pm 1\}$ such that $F_T\circ h = \varepsilon\cdot \delta^K$ and $\|F_T(s)\|<1$ for $s\in S\setminus h(K)$. Define $v\colon C(K)\to C(S)$ by $v(f)= T(\varepsilon\cdot f)$ for every $f\in C(K)$. Then, it is easy to check that $F_v\circ h =\delta^K$, thus $v$ is a norm-one linear extension operator for $h$. The fact that $h(K)$ is a $G_\delta$-set in $S$ follows from Corollary \ref{coro_zerosets} above.

(ii)$\Rightarrow$(i). Assume now that $h\colon K\to S$ is a continuous embedding admitting a norm-one extension operator $T\colon C(K)\to C(S)$ and such that $h(K)$ is a $G_\delta$ set of $S$. By using Theorem~\ref{theorem_cks} with $\varepsilon$ the constant function $1$, we only need to check that $\|F_T(s)\|<1$ for every $s\in S\setminus h(K)$. However this last condition may be not satisfied, for instance if $T$ is regular ---see~\cite[Definition 2.1]{Pelczynski1968}. 
Now, since $h(K)$ is a $G_\delta$ set of $S$, it is a zero-set of $S$ by Lemma~\ref{lema_same_Gd-zs}. Pick $g\in C(S)$ whose range is within $[0,1]$ and such that $h(K)=\{s\in S:\  g(s)=1\}$. Define $v\colon C(K)\to C(S)$ by $v(f):=g\cdot T(f)$ for every $f\in C(K)$. Now, given $s\in S$ and $f\in C(K)$, it holds
\[
\langle f, v^*(\delta_s^S)\rangle =\langle v(f), \delta_s^S\rangle =\langle g\cdot T(f), \delta_s^S\rangle= g(s)\langle T(f), \delta_s^S\rangle = g(s)\langle f, T^*(\delta_s^S)\rangle,
\]
so $v^*(\delta_s^S)=g(s)T^*(\delta_s^S)$ for every $s\in S$. In case $s=h(t)$ for some $t\in K$, then
\[
F_v(h(t))=g(h(t))\delta_t^K=\delta_t^K,
\]
since $g(h(t))=1$. In other case, for $s\in S\setminus h(K)$ it holds that $0\leq g(s)<1$, so
\[
\|F_v(s)\|=\|g(s)F_T(s)\|=|g(s)|\|F_T(s)\|<1,
\]
since $T$ is norm-one. Now, an appeal to Theorem~\ref{theorem_cks} shows that $v$ is a U-embedding and the proof is over.
\end{proof}

\begin{definition}
Let $S$ be a compact Hausdorff space. We say that a subset $K\subset S$ is a {\bf retract} of $S$ if there exists a continuous onto mapping $r\colon S\longrightarrow K$ such that $r(k)=k$ for every $k\in K$.
\end{definition}

The following result should be compared with Lemma~\ref{lemma_compositionOperator}.

\begin{proposition}\label{prop:retraction}
Let $S$ be a compact Hausdorff space and $K\subset S$ a retract which is a zero set. Then, there exists a U-embedding from $C(K)$ into $C(S)$.
\end{proposition}

\begin{proof}
Let $r\colon S\to K$ be a retraction and $f\in C(S)$ with $K=f^{-1}(\{0\})$. We can assume without loss of generality that $0\leq f(s)\leq 1$ for every $s\in S$. Set $F\colon S \to (M(K),w^*)$ defined for every $s\in S$ by 
\begin{equation*}
    F(s)= \dfrac{1-f(s)}{1+f(s)} \delta_{r(s)}.
\end{equation*}
Since $r$ is a retraction, we have $F(k) = \delta_{k}$ for every $k\in K$. If $s\in S\setminus K$ then $0\leq \|F(s)\|<1$. Theorem~\ref{theorem_cks} concludes that the associated isomorphism $T_F\colon C(K) \rightarrow C(S)$ is a U-embedding.
\end{proof}

\begin{remark}
\rm Proposition \ref{prop:retraction} above could also be obtained as a combination of \cite[Proposition 3.3]{Pelczynski1968} and Theorem~\ref{theorem_equiv-with-eos}.
\end{remark}

An easy corollary of the former proposition is the following:

\begin{corollary}
Let $\alpha<\kappa$ be two ordinals. Then $C([0,\alpha])$ can be U-embedded into $C([0,\kappa])$.
\end{corollary}
\begin{proof}
$[0,\alpha]$ is a retract of $[0,\kappa]$ and $[0,\alpha]$ is clearly a zero set in $[0,\kappa]$.
\end{proof}

\begin{corollary}
Let $h\colon K\to S$ be a neighborhood coretraction ---i.e., for each $t\in K$ there is a closed neighborhood of $t$, say $V\subset K$, such that $h(V)$ is a retract of $S$---. Then,  $C(K)$ can be U-embbeded into $C(S)$.
\end{corollary}
\begin{proof}
This is an straightforward consequence of \cite[Proposition 3.5]{Pelczynski1968} and Theorem~\ref{theorem_equiv-with-eos}.
\end{proof}

\begin{remark}
\rm
Theorem~\ref{theorem_equiv-with-eos} allows us to translate some of the results in \cite{Pelczynski1968} in terms of U-embeddings between $C$-spaces. For instance, we can show \cite[Proposition 2.9]{Pelczynski1968} that $C(K,\mathbb{C})$ can be U-embedded into $C(S,\mathbb{C})$ if and only if $C(K)$ is U-embedded into $C(S)$. \eor
\end{remark}

\subsection{Some applications}

The following result is a corollary of Theorem~\ref{theorem_cks}. It can be also deduced from Corollary~\ref{coro-one-dimension} and the fact that the points of G\^ateaux smoothness in $C(K)$ are peaking functions. Notice that all norms on a one-dimensional Banach space agree modulus a non-negative factor. Thus, we may assume that the finite-dimensional Banach space is $\R$ endowed with the absolute-value norm.

\begin{corollary}\label{coro_onedimension}
A one-dimensional Banach space can be U-embedded into certain $C(K)$ if and only if $K$ contains at least a $G_\delta$-point.
\end{corollary}
\begin{proof}
Theorem~\ref{theorem_cks} states that $\mathbb{R}=C(\{1\})$ can be U-embedded into $C(K)$ if and only if there is a function $f\colon K\to [-1,1]$ such that $\{t: |f(t)|=1\}$ has cardinality exactly $1$. This is equivalent to say that there exists $t\in K$ such that $\{t\}$ is a zero-set of $C(K)$. Now, a singleton in a completely regular space is a zero-set if and only if it is a $G_\delta$-set. The proof is over.
\end{proof}

A simple application of the previous result is given by the following corollary. It provides a necessary condition for U-embeddability.

\begin{corollary}
Let $K$ and $S$ be two compact spaces such that $C(K)$ is U-embedded into $C(S)$. If $K$ has at least a $G_\delta$-point then so does $S$.
\end{corollary}
\begin{proof}
This is a consequence of Lemma \ref{lemma-chain} and the previous corollary.
\end{proof}

It is known that a Valdivia compact $K$ can be constructed without $G_\delta$-points (see \cite{Correa}),
which leads to the fact that for any compact $S$ with at least a $G_\delta$-point (say, e.g., a metrizable or a convex space), the Banach space $C(S)$ can not be U-embedded into such $C(K)$ space. 

\medskip

Corollary~\ref{coro_onedimension} above can be generalized to a finite number of $G_\delta$-points. Corollary~\ref{coro_zerosets} gives one direction. Namely, if $\ell_\infty^{n}$ can be U-embedded into $C(K)$ then, $K$ contains at least $n$ $G_\delta$-points. The reciprocal is also true:

\begin{proposition}
Let $N$ be a natural number and $K$ be a compact Hausdorff space. Then the space $\ell_\infty^N$ can be U-embedded into $C(K)$ if, and only if, $K$ contains at least $N$ $G_\delta$-points.
\end{proposition}
\begin{proof} As we mentioned, one direction is given by Corollary \ref{coro_zerosets} above. 

Let us assume now that $K$ contains at least $N$ $G_\delta$-points $\{p_i\}_{i=1}^N\subset K$. We are going to construct a U-embedding from $C(\{p_i\}_{i=1}^N)$ ---which is isometrically isomorphic to $\ell_\infty^N$--- into $C(K)$. Choose for each of the points a family $\{G_n^i\}_{n=1}^\infty$ of decreasing open sets such that $\{p_i\}=\bigcap_n G_n^i$ for every $1\leq i\leq N$. Due to the fact that $K$ is Hausdorff, we can also assume that there exist a collection of pairwise disjoint open sets $\{U^i\}_{i=1}^N$ such that $G_n^i\subset U^i$ for every $n\geq 1$ and for every $1\leq i\leq N$. By the Urysohn lemma, there exists a continuous function $g_n\colon K\to[0,1]$ such that $g_n(p_i)=1$ for each $1\leq i\leq N$ and $g_n(t)=0$ for every $t\notin \bigcup_i G_n^i$. Using this function, define $h_n\colon K\to B_{\ell_1^N}$ by $h_n(t)=g_n(t)\delta_{p_i}$ whenever $t$ belongs to $G_n^i$ and $0\in\ell_1$ whenever $t$ does not belong to the union $\bigcup_i G_n^i$. It is easy to see that this function is continuous (it can be seen to be continuous in each of $N+1$ closed sets covering $K$). Now consider the function $h=\sum_n\frac{1}{2^n}h_n$. It is a well-defined function from $K$ to $B_{\ell_1}$ and it is continuous ---the partial sums are uniformly convergent over $K$, and we assume the topology of the norm on the unit ball. Now, observe that $\|h(t)\|=1$ if and only if $t$ belongs to $\{p_1,\dots,p_N\}$ in which case $h(p_i)=\delta_{p_i}$ for every $1\leq i\leq N$. By Theorem~\ref{theorem_cks}  the proof is over.
\end{proof}

\section{Factorization of U-embeddings}

The following result has a simple proof, that shall be omitted. 
\begin{proposition}\label{prop:minimal-ideal}
Let $T\colon X\to C(K)$ be a U-embedding. Put $Z:=F_T^{-1}(0)$ and let $I_K(Z)$ be the set of all elements in $C(K)$ that vanish on $Z$. Then, $T(X)\subset I_K(Z) \subset C(K)$, and the set $I_K(Z)$ is the minimal closed ideal of $C(K)$ containing $T(X)$.
\end{proposition}

\begin{remark}
\rm
It is expected that most of the time the set $F_T^{-1}(0)$ should be empty. If this is the case, as  $I_K(\emptyset)=C(K)$, we get that there is no proper closed ideal of $C(K)$ containing the range of $T$.\eor
\end{remark}

\begin{remark}\label{rem-minimal}
\rm The following observations is almost trivial: Let $K$ be a compact topological space, and let $K_0$ be a nonempty closed subspace. The restriction mapping $R:C(K)\rightarrow C(K_0)$ (i.e., $Rf:=f|_{K_0}$ for $f\in C(K)$) is linear and continuous. Then, $R^*:M(K_0)\rightarrow M(K)$ is one-to-one. Indeed, let $\nu_i\in M(K_0)$, $i=1,2$, be such that $\mu:=R^*(\nu_1)=R^*(\nu_2)$. Let $f_0\in C(K_0)$. Tietze's extension theorem ensures the existence of $f\in C(K)$ such that $Rf=f_0$. Thus,
$$
\langle f_0,\nu_i\rangle=\langle Rf,\nu_i\rangle=\langle f,R^*\nu_i\rangle=\langle f,\mu\rangle,
$$
for $i=1,2$. It follows that $\nu_1=\nu_2$.

This applies to the following situation: Assume that (i) $X$ is a Banach space, (ii) $K$ is a compact topological space, and (iii) there is a linear isometry $T:X\rightarrow C(K)$. As above, let $K_0$ be a nonempty closed subset of $K$. The mapping $T_0:=R\circ T:X\rightarrow C(K_0)$ is linear and continuous, and obviously $T_0^*:M(K_0)\rightarrow X^*$ satisfies $T^*_0=T^*\circ R^*$. Thus, if $T$ is a U-embedding it is tempting to say that $T_0$ is also a U-embedding (see Proposition \ref{prop:U_homeoT^*} above). There is a gap in the argument: $T_0$ may fail to be an isometric embedding. If this happens (for example, if $T^*(\pm K_0)$ contains $\Ext B_{X^*}$), then the conjecture turns out to be true.    \eor
\end{remark}

Remark \ref{rem-minimal} applies, in particular. to the following simple result, that shows the interest in the concept of the U-core, introduced in Definition \ref{def:U-core}, and shows that it is, in a sense, minimal for the ``factorization'' of a linear isometric embedding $T:X\rightarrow C(K)$ through a space $C(K_0)$:

\begin{proposition}\label{prop:minimal-ucore}
Let $T\colon X \to C(K)$ be a U-embedding. Take $G:= F_{|\ol{\Az}}$. Then, $T_G\colon  X \to C(\ol{\Az})$ is a U-embedding. 
\end{proposition}

At the end of Section \ref{sec:U-embeddings} we briefly mentioned that a kind of ``sandwich'' property for U-embeddings is likely to be false (more precisely, that it fails in general that given three Banach spaces $X\subset Y\subset Z$, $Y$ is U-embedded in $Z$ whenever $X$ is U-embedded in $Z$ and in $Y$). In this direction, Proposition \ref{prop:minimal-ideal} can be somehow understood as a partial positive statement for $Y$ being a closed ideal, also satisfying that the triplet $T(X)\subset Y \subset Z$ has ``pairwise'' property U. Thus, Proposition \ref{prop:minimal-ucore} shows that somehow a U-embedding $T$ of $X$ into $C(K)$ factorizes throughout a minimal $C(S)$ space (in this case, $S$ being the closure of the \ucore \ $\Az\subset K$). One may be tempted to think that this can  indeed be a factorization, i.e., $X$ is U-embedded in $C(\ol{\Az})$ and this is again U-embedded in $C(K)$ through the inclusion. Unfortunately $C(\ol{\Az})$ is always a quotient of $C(K)$, but it might not be possible to find a U-embedding into $C(K)$.

\begin{example}\label{ex:counterex-ideals}
\rm 
Let $K$ be a compact space having a non-dense closed subset $B\subset K$ which is not a zero-set. Take $A= \ol{K\setminus B}$, which is non-empty because $B$ is not dense. Thus, we have a U-embedding $T\colon X\to C(K)$, where $X=I_A$ is U-embedded in $C(K)$ through the inclusion $T=i$, and in this case the \ucore \ is $\Az = K\setminus A = B$. By Proposition \ref{prop:minimal-ucore}, we have that $I_A$ is indeed U-embedded in $C(\ol{\Az})= C(B)$. However, it is impossible to find a U-embedding of $C(B)$ into  $C(K)$ because $B$ is not a zero-set (see Corollary \ref{coro_zerosets}).
\end{example}

\section*{Acknowledgements}

Instituto Universitario de Matem\'atica Pura y Aplicada. Universitat Polit\`ecnica de Val\`encia (Spain). The authors were partially supported by the Universitat Polit\`ecnica de Val\`encia (Spain) and by grant PID2021-122126NB-C33 funded by MCIN/AEI/10.13039/501100011033 and by “ERDF A way of making Europe”. The first author was also supported by Generalitat Valenciana (through Project PROMETEU/2021/070 and the predoctoral fellowship CIACIF/2021/378) and by MCIN/AEI/10.13039/501100011033 (through Project PID2019-105011GB). The second author was also supported by Fundación Séneca, Región de Murcia (Grant 19368/PI/14). The third author was also supported by AEI/FEDER (Project MTM2017-83262-C2-1-P of Ministerio de Economía y Competitividad).

\printbibliography

\end{document}